\documentclass[american]{amsart}
\usepackage{amssymb, amsmath, amsfonts, amsthm, esint}
\usepackage[margin=1in]{geometry} 
\usepackage[utf8]{inputenc}
\usepackage{bbm}

\usepackage{cite}
\usepackage[colorlinks,linkcolor = blue,citecolor = green]{hyperref}

\usepackage{color}
\usepackage{comment}

\newcommand{\abs}[1]{\left\vert#1\right\vert}
\newcommand{\Abs}[1]{\big\vert#1\big\vert}
\newcommand{\norm}[1]{\left\|#1\right\|}  
\newcommand{\set}[1]{\left\{ #1 \right\}}
\newcommand{\brak}[1]{\left\langle #1 \right\rangle}

\newcommand{\dd}{\, {\rm d}}
\newcommand{\R}{\ensuremath{{\mathbb R}}}
\newcommand{\weak}{\ensuremath{{\rightharpoonup}}}
\renewcommand{\S}{\ensuremath{{\mathcal S}}}
\newcommand{\N}{\ensuremath{{\mathbb N}}}

\DeclareMathOperator{\eps}{\varepsilon}
\DeclareMathOperator{\embeds}{\hookrightarrow}

\newcommand{\beq}{\begin{equation}}
\newcommand{\eeq}{\end{equation}}
\newcommand{\beqs}{\begin{equation*}}
\newcommand{\eeqs}{\end{equation*}}
\newcommand{\bal}{\begin{equation}\begin{aligned}}
\newcommand{\eal}{\end{aligned}\end{equation}}
\newcommand{\bals}{\begin{equation*}\begin{aligned}}
\newcommand{\eals}{\end{aligned}\end{equation*}}

\newcounter{num} \numberwithin{num}{section}

\newtheorem{theorem}[num]{Theorem}
\newtheorem{proposition}[num]{Proposition}
\newtheorem{lemma}[num]{Lemma}
\newtheorem{corollary}[num]{Corollary}
\theoremstyle{definition}

\theoremstyle{remark}
\newtheorem{remark}[num]{Remark}

\numberwithin{equation}{section}

\author{William Golding, Maria Pia Gualdani, Am\'elie Loher}
\title[Nonlinear regularization estimates]{Nonlinear regularization estimates and global well-posedness for the Landau-Coulomb equation near equilibrium}
\date{\today}							
\address[William Golding]{\newline Department of Mathematics, \newline The University of Texas at Austin, Austin, TX 78712, USA}
\email{wgolding@utexas.edu}
\address[Maria Pia Gualdani]{\newline Department of Mathematics, \newline The University of Texas at Austin, Austin, TX 78712, USA}
\email{gualdani@math.utexas.edu}
\address[Am\'{e}lie Loher]{\newline Department of Pure Mathematics and Mathematical Statistics, \newline University of Cambridge, Cambridge, CB3 0WA, United Kingdom}
\email{ajl221@cam.ac.uk}

\thanks{\textbf{Acknowledgements:} W. Golding is partially supported by NSF grant DMS 1840314. M. Gualdani is partially supported by DMS-2206677 and DMS-1514761. A. Loher is funded by the Cambridge Trust and Newnham College scholarship. The authors would like to thank the Isaac Netwon Institute at the University of Cambridge for their kind hospitality.}

\begin{document}

\begin{abstract}
 We consider the Landau equation with Coulomb potential in the spatially homogeneous case.  We show short time propagation of smallness in $L^p$ norms for $p>3/2$ and instantaneous regularization in Sobolev spaces. This yields new short time quantitative a priori estimates that are unconditional near equilibrium. We combine these estimates with existing literature on global well-posedness for regular data to extend the well-posedness theory to small $L^p$ data with $p$ arbitrarily close to $3/2$. The threshold $p = 3/2$ agrees with previous work on conditional regularity for the Landau equation in the far from equilibrium regime. 
 
In light of the monotonicity of the Fisher information shown in the recent preprint [arXiv:2311.09420], our primary nonlinear regularization estimate holds even in the far-from-equilibrium regime. As a consequence, we obtain exponential convergence to equilibrium for suitably localized solutions in every Sobolev norm.
 
\end{abstract}

\maketitle

\tableofcontents

\section{Introduction}
The Landau-Coulomb equation is a fundamental model in plasma physics, describing the statistical behavior of particles in a collisional plasma. Despite its widespread use, the mathematical understanding of this equation has been limited, particularly with regards to the existence of global-in-time smooth solutions. In this paper, we consider the spatially homogeneous version of the Landau-Coulomb, written as
\begin{equation}
	\partial_t f(t, v) = \mathcal{Q}(f,f)(t, v) \qquad (t, v) \in (0, \infty) \times \R^3,
\label{eq:landau1}
\end{equation}
where $f:\R^+ \times \R^3 \rightarrow \R^+$ is the unknown distribution and $\mathcal Q$ is the collision operator given by 
\begin{equation}\label{eq:q}
	\mathcal Q(f,f)(v) := \frac{1}{8\pi}\nabla_v \cdot \left(\int_{\R^3} \frac{\Pi(v - v_*)}{\abs{v-v_*}} \left\{f_*\nabla f - f \nabla f_*\right\}\dd v_*\right), \qquad \Pi(z) := \rm{Id} - \frac{z \otimes z}{\abs z^2},
\end{equation}
with the standard notation $f = f(v), ~ f_* = f(v_*)$. While \eqref{eq:landau1} is not as physically relevant as the full inhomogeneous model, the study of \eqref{eq:landau1} already provides several mathematical challenges, but is more tractable. Indeed, \eqref{eq:landau1} may be rewritten as the following quasi-linear, divergence-form parabolic equation with non-local coefficients:
\begin{equation}
	\partial_t f(t, v) = \nabla \cdot \left(A[f]\nabla f - \nabla a[f] f\right),
\label{eq:landau}
\end{equation}
where the coefficients $A$ and $a$ are given by 
\begin{equation}\label{eq:coefficients}
A[f] = \frac{\Pi(v)}{8\pi |v|} \ast f, \qquad a[f] = (-\Delta)^{-1}f = \frac{1}{4\pi |v|} \ast f.
\end{equation}
It is also possible to rewrite \eqref{eq:landau} in non-divergence form as:
\begin{equation*}
\partial_t f(t,v) = A[f]:\nabla^2 f + f^2,
\end{equation*}
which highlights the main challenge in studying this equation, namely the competition between the reaction term, which is known to cause blow-up, and the nonlinear, non-local diffusion term. 

Complete existence theory for solutions to \eqref{eq:landau1} has been established only for weak solutions. In \cite{VillaniHsols}, Villani constructed global-in-time very weak solutions (often referred to as H-solutions) to \eqref{eq:landau1} for general initial datum. Subsequently, in \cite{desvillettesHsols}, Desvillettes used propagation of $L^1$ moments to clarify that Villani's solutions satisfy \eqref{eq:landau1} in a standard weak sense. The moment estimates were then improved in \cite{DCH} and used to provide quantitative rates of decay to equilibrium in $L^1$. Notably, the question of uniqueness of these solutions remains open, even for regular initial datum.

The uniqueness of weak solutions with bounded mass, momentum, energy, and entropy which further lie in $L^1(0,T;L^\infty)$ was established by Fournier in \cite{Fournier} using probabilistic methods. On the other hand, the authors of \cite{ChernGualdani} relaxed the boundedness assumption to $L^\infty(0,T;L^p)$ for $p > 3/2$, assuming the initial datum decays sufficiently fast. However, these results do not necessarily apply to the global solutions constructed by Villani, as it is currently unknown whether general weak solutions belong to either integrability class. 

Few unconditional regularity results exist for general weak solutions of \eqref{eq:landau}. The authors of \cite{PartialRegularity1} showed partial regularity in time through the use of a De Giorgi-type iteration, which takes advantage of the control of entropy to determine a dimensional bound on the set of singular times. The results were then extended in \cite{PartialRegularity2} to provide a bound on the dimension of the set of singular points in space-time. In a separate work, \cite{DHJ} demonstrated the existence of a monotone functional that leads to the eventual smoothness of general weak solutions. If the initial value is  close to the equilibrium in weighted $H^1$ space, the authors were able to show global-in-time well-posedness. Furthermore, the article \cite{KleberMischler} discusses a global-in-time existence result for the inhomogeneous Landau equation, assuming closeness to the equilibrium in a weighted $H^1_xL^2_v$-space. The proof in \cite{KleberMischler} uses spectral methods, which are significantly more complicated for general $L^p$ spaces when $p \neq 2$. We would also like to mention the related result in \cite{KimGuoHwang}: there, the authors  proved that smooth solutions for inhomogeneous Landau-Coulomb  equation exist globally in time for initial data close to equilibrium in a weighted $L^\infty$ sense. The method relies upon pointwise estimates using a De Giorgi iteration to close their estimates. 

While unconditional regularity results remain elusive, various conditional results have been established based on classical elliptic-parabolic iteration methods that highlight the smoothing effects of the diffusion in \eqref{eq:landau}. In \cite{GualdaniGuillen1}, the second author and Guillen used barrier arguments to obtain pointwise bounds for radially symmetric solutions, conditional to the $L^\infty(0,T; L^p)$ norm with $p > 3/2$. This is a strong enough estimate to build global-in-time solutions to a related model, the isotropic Landau equation introduced in \cite{KriegerStrain} (see also \cite{GressmanKriegerStrain}), but not for \eqref{eq:landau}. In \cite{Silvestre}, Silvestre treated \eqref{eq:landau} as a non-divergence form parabolic equation and applied arguments inspired by Krylov-Safanov to deduce a priori pointwise estimates for smooth solutions to \eqref{eq:landau}. This provides unconditional bounds in the case of moderately soft potentials, but only provides bounds conditional on a weighted $L^\infty(0,T;L^p)$-norm for $p > 3/2$ of the solution for Coulomb potential. In a subsequent work, \cite{GualdaniGuillen2}, the authors used the theory of $A_p$ weights and a Moser-type iteration to obtain pointwise bounds for weak solutions, provided solutions belong to $L^\infty(0,T;L^p)$-norm for $p > 3/2$. 

\medskip

The purpose of this paper is to show unconditional regularization effects and pointwise bounds when initial data are close to equilibrium in $L^p$ with $p$ arbitrarily close to $3/2$. Such unconditional, quantitative estimates are strong enough to prove existence of  global in time classical solutions. This provides the first proof of global well-posedness near equilibrium for initial data in low $L^p$ spaces and without smallness in weighted norms. Previous results require initial data near equilibrium in a weighted $L^2$ space \cite{KleberMischler}, a weighted $H^1$ space \cite{DHJ}, or a weighted $L^\infty$ space \cite{KimGuoHwang}. The restriction $p > 3/2$ is essential to our arguments, which hinge upon short-time smoothing estimates. Our smoothing estimates hold when the diffusion dominates on short time scales. This occurs in \emph{subcritical} norms; the coefficient $A[f]$ in \eqref{eq:coefficients} is bounded for $f \in L^1 \cap L^p$ with $p> \frac{3}{2}$.
Moreover, by Gagliardo-Sobolev-Nirenberg inequality, the reaction term for $g\ge 0$ can be bounded as 
\begin{equation*}
\begin{aligned}
\mathrm{REACTION} = \int_{\R^3} g^{p+1} \dd v \lesssim   \left(\int_{\R^3} g^{3/2} \dd v\right)^{2/3} \int_{\R^3} \abs{\nabla g^{p/2}}^2 \dd v = \norm{g}_{L^{3/2}}  \times \mathrm{``DIFFUSION"}.
\end{aligned}
\end{equation*}
This shows that the diffusion will control the reaction if $\norm{g}_{L^{3/2}}$ is small. We will derive an inequality similar to the preceding one for $g = f-\mu$ with the Landau diffusion term (which involves weights), and use it to quantitatively track the time evolution of the $L^p$ norm for a short time. An iteration argument inspired by the De Giorgi method will yield smoothing estimates. Qualitatively, this implies that if the initial data is small in some $L^p$ spaces ($p>3/2$), the corresponding solution will be regular and small (in all Sobolev norms) after some time passes. Uniqueness will allow us to piece together the local-in-time solution constructed here with the global-in time solution for small smooth data constructed in \cite{DHJ}, or alternatively in \cite{KleberMischler} or \cite{KimGuoHwang}. This yields a single global in time solution for small $L^p$ data with $p$ arbitrarily close to $3/2$. Let us also stress that we require smallness only for $f$, not for moments of $f$.

Before presenting our main result, we first establish some notation and revisit the some basic properties of \eqref{eq:landau1}.
We use the bracket convention $\langle v\rangle := (1+\abs{v}^2)^{\frac{1}{2}}$ and define the $L^p_m$ norm for $p \in [1, \infty)$ and $m \in \R$, as follows:
\begin{equation*}
    \norm{f}_{L^p_m(\R^3)}^p :=  \int_{\R^3} \abs{f(v)}^p\langle v\rangle^{m} \dd v, \qquad \norm{f}_{L^\infty_m(\R^3)} := \sup_{v \in \R^3} \abs{f(v)}\langle v \rangle^m.
\end{equation*}
We define 
the weighted Sobolev norm $H^k_m(\R^3)$  as
\begin{equation*}
    \norm{f}_{H^k_m(\R^3)}^2 := \sum_{\abs{\alpha} \leq k} \int_{\R^3} \Abs{\partial^\alpha\big(\langle v\rangle^{\frac{m}{2}} f\big)}^2 \dd v.
\end{equation*}
Next, we recall that \eqref{eq:landau} formally conserves mass, momentum, and energy:
\begin{equation*}
    \frac{\dd}{\dd t}\int_{\R^3} \begin{pmatrix} 1 \\ v \\ |v|^2 \end{pmatrix} f(t,v) \;\dd v = 0.
\end{equation*}
Moreover, solutions to \eqref{eq:landau} have decreasing Boltzmann entropy:
\begin{equation*}
    \frac{\dd}{\dd t}\int_{\R^3} f \log(f) \;\dd v \le 0.
\end{equation*}
Consequently, the steady states of \eqref{eq:landau} may be characterized explicitly: $f \in L^1_2$ is a steady state of \eqref{eq:landau} if and only if $f$ is a Maxwellian distribution, given by
\beqs
    \mu_{\rho, u, \theta}(v) = \frac{\rho}{(2\theta\pi)^{3/2}}e^{-\frac{\abs{v-u}^2}{2\theta}},
\eeqs
where $\rho \geq 0$ is the mass, $u \in \R^3$ is the mean velocity, and $\theta >0$ is the temperature of the plasma. These parameters are constant for any reasonable solution of \eqref{eq:landau} due to conservation of mass, momentum, and energy. Further, because \eqref{eq:landau} is translation invariant and enjoys a two-parameter scaling invariance, we will assume that our initial datum $f_0(v)$ satisfies the normalization
\begin{equation}\label{eq:normalization}
    \int_{\R^3} \begin{pmatrix} 1 \\ v \\ |v|^2\end{pmatrix} f_0(v) \;\dd v = \begin{pmatrix} 1 \\ 0 \\ 3\end{pmatrix},
\end{equation}
which fixes the corresponding Maxwellian with same mass, momentum and energy as
\begin{equation*}
    \mu(v) := (2\pi)^{-\frac{3}{2}}e^{-\frac{\abs{v}^2}{2}}.
\end{equation*}
We now present the main result: 
\begin{theorem}\label{thm:existence}
Fix $p > \frac{3}{2}$ and let $m$ be any positive number greater than 
$$
\max\left(\frac{9}{2}\frac{p-1}{p - \frac{3}{2}}\max\left(1,\frac{p(p- 3/2)}{p^2 - 2p + 3/2}\right),55\right).
$$
For any given $M$, $H$ and $f_{in} \in L^1_m 
\cap L^p(\R^3)$ satisfying 
 \begin{equation*}
        \norm{f_{in}}_{L^1_m} \le M \qquad \text{and} \qquad \int_{\R^3} f_{in} \abs{\log(f_{in})} \;dv \le H, 
    \end{equation*}
there exists $\delta = \delta(p,m, M,H) > 0$ such that if 
$$
 \int_{\R^3} \abs{f_{in} - \mu}^p \le \delta,
$$
equation \eqref{eq:landau} admits a unique global classical solution on $(0,\infty) \times \R^3$ with initial datum $f_{in}$. This solution satisfies, for $t>0$, the smoothing estimate
\begin{equation}\label{inst_reg}
\|f(t)\|_{L^\infty} \le C(p,m,M,H)\left(1 + t^{-\frac{1}{1 + \gamma}}\right),
\end{equation}
for some $\gamma >0$ that only depends on $p$ and $m$:
\beqs
    \gamma := \frac{2(p-3/2)}{3m}\left[m - \frac{9}{2}\frac{p-1}{p-3/2}\right] \in \left(0, \frac{2}{3}\left(p - \frac{3}{2}\right)\right). 
\eeqs
\end{theorem}
Theorem \ref{thm:existence} provides further evidence that all $L^p$ norms for $p > 3/2$ are sub-critical in the sense that diffusive effects dominate. Combined with the conditional regularity results in \cite{Silvestre, GualdaniGuillen1, GualdaniGuillen2} mentioned above, this emphasizes the role of the $L^{3/2}$ norm as a sort of critical norm for the equation. If one relies on the analogy to semi-linear parabolic problems, one could expect blow-up in $L^p$ norms for $1 < p < 3/2$, since the semi-linear heat equation is known to cause blow-up in finite time for super-critical norms (e.g. $L^p$ with $p < 3/2$), see \cite{GigaKohn}. Finite time blow-up in $L^p$ norms for $1 < p < 3/2$, however, seems unrealistic for \eqref{eq:landau}, since the $L^1$ norm is conserved. Theorem \ref{thm:existence} also clarifies when the analogy of \eqref{eq:landau} to semi-linear equations remains valid and provides useful intuition. On the other hand, we believe that Theorem  \ref{thm:existence} should be true also for $p$ arbitrarily close to one, but the proof seems at the moment out of reach.


Since we are in the sub-critical regime, heuristically, the solution obtained in Theorem \ref{thm:existence} is expected to behave similarly to solutions to the heat equation, with a correction that takes into account the degeneracy of the diffusion coefficient $A[f]$ at large velocities. In fact, ignoring the $m$ dependence of the constant $C$ and letting $m \to \infty$ in the exponent $\gamma$, the smoothing estimate in (\ref{inst_reg}) approaches the well-known $L^p-L^\infty$ estimate for the heat equation on $\R^3$:
\begin{equation*}
    \norm{e^{t\Delta}f_{in}}_{L^\infty} \le C(p)t^{-\frac{3}{2p}}\|f_{in}\|_{L^p}.
\end{equation*}
The nonlinear dependence of the small parameter $\delta$ from the size of the initial data in  $L^1_m$ and $L \log L$ should be interpreted as follows:  if a sequence of initial data $f_{in}^k$ converges to $\mu$ in $L^p$ and remains equi-bounded in $L^1_m$ and $L \log L$, then for sufficiently large $k$, these initial data will admit global classical solutions according to Theorem \ref{thm:existence}.

Unlike the results of \cite{DHJ}, \cite{KimGuoHwang} and \cite{KleberMischler}, Theorem \ref{thm:existence} does not require smallness for the moments of the initial data. This distinction is subtle but non-trivial. Our method uses weighted spaces only to quantify the strength of the diffusion in our energy estimates, not to bound our energy functional. This is possible thanks to several weighted Sobolev inequalities combined with moment estimates.

After the completion of this manuscript, a breakthrough result appeared in \cite{GuillenSilvestre}, where the authors show monotonicity of the Fisher information for smooth solutions to \eqref{eq:landau} with Gaussian decay at infinity. Since the Fisher information directly controls the $L^3$ norm, a version of the non-linear regularization estimate in Theorem \ref{thm:existence} for $p \leq 3$ applies even in the far-from-equilibrium regime and uniformly in time. Indeed, combining the $p = 2$ regularization estimate of Proposition \ref{prop:degiorgi} with the monotonicity of the Fisher information in \cite[Theorem 1.2]{GuillenSilvestre} and the convergence to equilibrium in weighted $L^1$ norms \cite[Theorem 2]{DCH}, we are able to obtain long-time asymptotics for a large class of solutions in strong norms:
\begin{corollary}\label{cor:long_time}
    Let $f_{in}$ satisfy the conditions of \cite[Theorem 1.2]{GuillenSilvestre}; in particular, $f_{in} \in C^1$, has bounded Fisher information and is bounded above by a Maxwellian. Then, there is a constant $C_1 > 0$ depending only on the Fisher information of the initial data such that for any $1 < t$ and $m > 9$,
    \begin{equation}\label{eq:long_time1}
        \norm{f(t)-\mu}_{L^\infty} \leq C_1 \norm{f(t)-\mu}_{L^1_m}^{\frac{9}{7m-9}}.
    \end{equation}
    In particular, for some constants $C_2,\,\lambda_0 > 0$ depending only on the initial Fisher information and initial $L^1(\exp(\frac{1}{2}\brak{v}^{1/2})\dd v)$ norm, 
    \begin{equation}\label{eq:long_time2}
        \norm{f(t)-\mu}_{L^\infty} \leq C_2 \exp\left(- \lambda_0 \frac{t^{1/7}}{(\log(1+t))^{6/7}}\right).
    \end{equation}
\end{corollary}

We remark that the uniform bound on the Fisher information of $f$ implies that the $L^3$ norm of $f - \mu$ is uniformly bounded in time, but not necessarily small. Therefore, while the ODE argument used in Lemma \ref{lem:perthame} does not directly apply, our standard $L^2$ energy estimate used frequently below yields
\begin{equation*}
    \sup_{T-1 < t < T} \int_{\R^3} f^2(t) \dd v + \int_{T-1}^T\int_{\R^3} \brak{v}^{-3} \abs{\nabla f}^2 \dd v\dd t \le C\int_{T-1}^T\int_{\R^3} f^3 \dd v,
\end{equation*}
where the right hand side is bounded uniformly in time by the Fisher information of $f$. Consequently, the same quantity is bounded uniformly in time for $h = f - \mu$, which bounds the functional $E_0$ appearing in the estimate in Proposition \ref{prop:degiorgi}. Therefore, applying Proposition \ref{prop:degiorgi} with $p = 2$ and $m > 9$ on the time interval $[T-1,T]$---which can be done using translation invariance in time---one obtains the long-time bound \eqref{eq:long_time1}, where only moments appear on the right hand side. Applying the asymptotic behavior of the $L^1$ norm shown in \cite[Theorem 2]{DCH}, one deduces \eqref{eq:long_time2}. Using standard $L^2$ estimates and the asymptotic $L^\infty$ bound in \eqref{eq:long_time2}, the convergence also holds in $H^k_m$ for all $k \ge 0$ and all $m \ge 0$, with the same rate (see Lemma \ref{lem:higher_regularity} for the $k = 1$ case, which is easily generalized). Many variants of Corollary \ref{cor:long_time} are possible by applying the various convergence to equilibrium results present in the literature (e.g. \cite{KleberMischler} and the references therein).

\medskip

Let us state the precise definition of solution found in Theorem \ref{thm:existence}: the function $f:\R^+ \times \R^3 \rightarrow \R^+$ is a global solution to \eqref{eq:landau} with initial datum $f_{in}$ in the following sense:
\begin{itemize}
    \item For each $T > 0$, $f \in C(0,T; L^1 \cap L^p) \cap L^\infty(0,T; L^1_m) \cap L^1(0,T;L^\infty)$;
    \item For each $0 < t < T$, $f\in L^\infty(t,T;L^\infty \cap H^1_2)$;
    \item This $f$ solves \eqref{eq:landau} in the sense of distributions on $[0,\infty)\times \R^3$, i.e. for each  $\varphi \in C^\infty_c([0,\infty)\times \R^3)$,
    \begin{equation*}
        \int_{\R^3} f_{in}\varphi(0) \dd v - \int_0^\infty\int_{\R^3} f\partial_t \varphi \dd v \dd t = \int_0^\infty\int_{\R^3}\nabla^2 \varphi : A[f]f + 2\nabla \varphi \cdot \nabla a[f]f \dd v\dd t.
    \end{equation*}
    \item This $f$ is a classical solution of \eqref{eq:landau} on $(0,\infty) \times \R^3$ and satisfies the short time smoothing estimate (\ref{inst_reg}).
\end{itemize}

The nonlinear nature of \eqref{eq:landau} and the low regularity of the initial data $f_{in}$ cause some minor difficulties in stating precisely what notion of solution holds up to time $0$. To illustrate the main ideas of Theorem \ref{thm:existence} and provide the main step of its proof, we isolate the result of Theorem \ref{thm:existence} for Schwartz class initial data in the following proposition:
\begin{proposition}\label{prop:smooth_data}
    Fix $p$, $m$, $H$, $M$, and $f_{in}$ as in Theorem \ref{thm:existence} and suppose additionally that $f_{in} \in \S(\R^3)$. Then, there is an $\delta = \delta(p,m,H,M) > 0$ such that $\|f_{in} -\mu\|_{L^p} < \delta$ implies there is a unique global-in-time solution $f:\R^+ \times \R^3 \rightarrow \R^+$ to \eqref{eq:landau} satisfying $f\in C([0,\infty);L^p)$. Furthermore, $f$ is regular in the sense that $f \in C^\infty(\R^+;\S(\R^3))$.
\end{proposition}
For $h := f - \mu$ a Schwartz class solution, we show that  $h(t,v)$ is small in $L^\infty$ after some time $t_0$, where $t_0$ is explicitly computed. The method used here is a variation of an ODE argument introduced by Corrias and Perthame in \cite{PC} for the Keller-Segel model. Higher regularity estimates then imply the smallness of $h(t)$ in a weighted Sobolev norm for $t\in (t_0,3t_0)$. This quantitative control over higher regularity of $h$ enables us to apply a result from \cite{DHJ} (summarized in Theorem \ref{thm:desvillettes_he_jiang} below) to conclude that there is a global solution to \eqref{eq:landau} with initial datum $f(2t_0)$.   The uniqueness result of \cite{Fournier}  (or \cite{ChernGualdani}) allows us to piece together the local-in time solution constructed in the time interval $(0,2t_0)$ with the global solution on $(2t_0, +\infty)$ to obtain a single global-in-time solution with initial datum $f_{in}$. Finally,  we remove the Schwartz class assumption using a compactness argument and obtain Theorem \ref{thm:existence}. 

The manuscript is organized as follows: 
In Section \ref{sec:prelims}, we list known results that are central to our proof.  In Section \ref{sec:loc}, we analyze the evolution of $f - \mu$ in $L^p$ and show that the $L^p$ norm \emph{remains small} for a short time.  In Section \ref{sec:DeGiorgi}, we use the De Giorgi method to prove short time smoothing estimates. Qualitatively, this implies that if the initial datum has small $L^p$ norm, the solution will have small $L^\infty$ norm after some time passes.  In Section \ref{sec:smooth_data}, we prove Proposition \ref{prop:smooth_data}, which concludes the argument for Schwartz class initial data.  In Section \ref{sec:compactness}, we deduce Theorem \ref{thm:existence} from Proposition \ref{prop:smooth_data} using a compactness argument  to remove the Schwartz class assumption on the initial data.

\section{Preliminaries}
\label{sec:prelims}
\subsection{Preliminary results}
We collect here known results that are used without proof. We begin with a local-in-time well-posedness result of Henderson, Snelson, and Tarfulea \cite{HendersonSnelsonTarfulea}. We present the result in a simplified form, suitable for our purposes:
\begin{theorem}[\protect{\cite[Theorem 1.1 and Theorem 1.2]{HendersonSnelsonTarfulea}}]\label{thm:henderson_snelson_tarfulea}
Suppose $f_{in} \in \S(\R^3)$ is non-negative. Then, there is a time $0 < T^* \le \infty$ and a function $f:[0,T^*)\times \R^3 \rightarrow \R^+$ with $f\in C^\infty(0,T^*;\S(\R^3))$ such that $[0,T^*)$ is the maximal time interval on which $f$ is the unique classical solution to \eqref{eq:landau} with initial datum $f_{in}$. Furthermore, if $T^* < \infty$, then $f$ satisfies
\begin{equation*}
\lim_{t \nearrow T^*} \|f(t)\|_{L^\infty} = +\infty.
\end{equation*}
\end{theorem}

We continue with a global-in-time existence result for small data, which is a simplified version of a result shown by Desvillettes, He and Jiang in \cite{DHJ}: 
\begin{theorem}[\protect{\cite[Proposition 1.2]{DHJ}}]\label{thm:desvillettes_he_jiang}
Suppose $\|f_{in}\|_{L\log L} + \|f_{in}\|_{L^1_{55}} \le K$ and $f_{in}$ is a non-negative initial datum. Let $\mu$ be the corresponding Maxwellian. Then, there is an $\eps_1 = \eps_1(K) > 0$ and $C = C(K) > 0$ such that if 
\begin{equation*}
\norm{(\nabla f_{in} - \nabla\mu)\langle \cdot \rangle^2}_{L^2} + \norm{(f_{in} - \mu)\langle \cdot \rangle^2}_{L^2} < \eps_1,
\end{equation*}
then there is a global solution $f:\R^+ \times \R^3 \rightarrow \R^+$ to \eqref{eq:landau}, satisfying
\begin{equation*}
\norm{(\nabla f(t) - \nabla\mu)\langle \cdot \rangle^2}_{L^2} + \norm{(f(t) - \mu)\langle \cdot \rangle^2}_{L^2} \le \frac{C}{(1 + t)^{15/4}}.
\end{equation*}
\end{theorem}

We conclude this section by recalling a uniqueness result due to Fournier:
\begin{theorem}[\protect{\cite[Theorem 2]{Fournier}}]\label{thm:Fournier}
    Suppose $f,\ g:[0,T]\times \R^3 \rightarrow \R^+$ are solutions to \eqref{eq:landau} in the sense of distributions with the same initial datum $f_{in} \in L^1_2$ and both belong to $L^\infty(0,T;L^1_2) \cap L^1(0,T;L^\infty)$. Then, $f(t) = g(t)$ for all $0 < t < T$.
\end{theorem}

Very frequently in the rest of the manuscript we will use the following  weighted Lebesgue interpolation result, which is a simple consequence of H\"older's inequality.
\begin{lemma}\label{lem:interpolation}
Suppose $\theta \in (0,1)$, $1 \le q, p_1, p_2 \le \infty$, and $\beta, \alpha_1,\alpha_2 \in \R$ satisfy the relations
\begin{equation*}
\frac{1}{q} = \frac{\theta}{p_1} + \frac{1-\theta}{p_2} \quad \text{and} \quad \beta = \theta \alpha_1 + (1-\theta)\alpha_2.
\end{equation*}
Then, for any $f\in L^{p_1}_{\alpha_1} \cap L^{p_2}_{\alpha_2}$, $f\in L^q_\beta$ and
\begin{equation*}
\big\|\brak{v}^\beta f\big\|_{L^q} \leq \big\|\brak{v}^{\alpha_1}f\big\|_{L^{p_1}}^\theta\big\|\brak{v}^{\alpha_2}f\big\|_{L^{p_2}}^{1-\theta}.
\end{equation*}
\end{lemma}

The following weighted Sobolev-type inequality will be crucial in the proof of regularization.  The proof is in the Appendix.  
\begin{lemma}\label{lem:poincare}
Suppose $g: \R^3 \rightarrow \R$ is Schwartz class and let $1 \le s \le 6$. Then, there are universal constants $C_1(s)$ and $C_2(s)$ such that
\begin{equation*}
\left(\int_{\R^3} |g|^6 \brak{v}^{-9}\;dv\right)^{1/3} \le C_1\int_{\R^3} |\nabla g(v)|^2 \brak{v}^{-3} \;\dd v + C_2\left(\int_{\R^3} \abs{g}^s\dd v\right)^{2/s}.
\end{equation*}
\end{lemma}

The following lemma collects standard estimates for the coefficients $A$ and $\nabla a$ defined in \eqref{eq:coefficients}. The following lemma controls the degenerate diffusion associated to \eqref{eq:landau} and quantifies the parabolic nature associated to the Landau equation when written in the form \eqref{eq:landau}.
\begin{lemma}[From \protect{\cite[Lemma 3.1]{BedrossianGualdaniSnelson}}, \protect{\cite[Lemma 2.1]{GGZ}}, \protect{\cite[Lemma 3.2]{Silvestre}}]\label{lem:Alower}
Suppose $f\in L^1_2(\R^3)$ is non-negative, satisfies the normalization \eqref{eq:normalization}, and has finite Boltzmann entropy, $\int f\log(f) \le H$. Then, there is a constant $c_0 = c_0(H)$ such that $A[f]$ satisfies the pointwise bound,
\begin{equation*}
\abs{A[f]} \ge \frac{c_0}{1 + |v|^3}.
\end{equation*}
Furthermore, if $f\in L^p$ for $3/2 < p \le \infty$, then
\begin{equation*}
\|A[f]\|_{L^\infty} \le C\|f\|_{L^1}^{\frac{2}{3}\frac{p - \frac{3}{2}}{p - 1}}\|f\|_{L^p}^{\frac{1}{3} \frac{p}{p-1}} \qquad \text{and} \qquad \|\nabla a[f]\|_{L^3} \le C\|f\|_{L^1}^{\frac{2}{3}\frac{p - \frac{3}{2}}{p - 1}}\|f\|_{L^p}^{\frac{1}{3} \frac{p}{p-1}}.
\end{equation*}
Moreover, if $3 < p \le \infty$ and $1 < q < \infty$, then
\begin{equation*}
    \norm{\nabla a[f]}_{L^\infty} \le C\|f\|_{L^1}^{\frac{p-3}{3(p-1)}}\|f\|_{L^p}^{\frac{2p}{3(p - 1)}} \qquad \text{and} \qquad \norm{\nabla^2 a[f]}_{L^q} \le C \|f\|_{L^q}.
\end{equation*}
\end{lemma}

We will frequently use that \eqref{eq:landau} propagates $L^1$-moments of any order. More precisely, $L^1$-moments of any order $s > 2$ grow at most linearly in time.
\begin{lemma}[\protect{\cite[Lemma 2.1]{DHJ}}]\label{lem:moments_h}
Let $\ell > \frac{19}{2}$. Fix a non-negative initial datum $f_{in} \in L^1_\ell \cap L\log L$ such that $\int f_{in} \dd v = 1$, $\int f_{in} \abs{v}^2 \dd v = 3$. Suppose $f:\R^3 \rightarrow \R$ is any weak solution of \eqref{eq:landau} and let $h := f- \mu$. 
Then for all $\theta \in [0, \ell]$ and $ q< q_{\ell, \theta}$ with
\begin{equation}\label{eq:q_ltheta}
    q_{\ell, \theta} = -\frac{2\ell^2 - 25\ell +57}{18(\ell-2)} \Big(1-\frac{\theta}{\ell}\Big) + \frac{\theta}{\ell},
\end{equation}
there exists a constant $C_3 = C_3(\theta, \ell, K)$, where $K$ is such that $\norm{f_{in}}_{L^1_\ell} + \norm{f_{in}}_{L\log L} \leq K$,
such that 
\begin{equation}\label{eq:momentbound_h}
    \forall t \geq 0, \qquad \norm{h(t, \cdot)}_{L^1_\theta} \leq C_3(1+t)^q.
\end{equation}
\end{lemma}

\section{Local-in-Time Propagation of \texorpdfstring{$L^p$}{Lp} Norms}\label{sec:loc}

In this section, we study the evolution of $L^p$-norms for $\frac{3}{2} < p < \infty$ of $h = f - \mu$, where $f$ is a solution to \eqref{eq:landau} and $\mu$ is the corresponding Maxwellian. We show that the $L^p$-norms of initial datum $h_{in}$ are propagated by smooth solutions of \eqref{eq:landau}, at least for some short time interval $[0,T_0]$. Since we work with smooth solutions, the calculations in this section are rigorous estimates. We note that for $p = 2$ a version of the following estimate has previously appeared as an a-priori bound in \cite{AlexandreLiaoLin}. 

\begin{lemma}\label{lem:perthame}
    Fix $\frac{3}{2} < p < \infty$, $m > \frac{9}{2}\frac{p-1}{p - \frac{3}{2}}$, $H \in \R$ and $M \in \R^+$. Suppose $f_{in}$ is a non-negative Schwartz class function, satisfying the normalization \eqref{eq:normalization} and the bounds
    \begin{equation*}
        \norm{f_{in}}_{L^1_m} \le M \qquad \text{and} \qquad \int_{\R^3} f_{in} \abs{\log(f_{in})} \;dv \le H.
    \end{equation*}
    Let $f:[0,T^*)\rightarrow \R^+$ be the unique Schwartz class solution to \eqref{eq:landau} with initial datum $f_{in}$ on its maximal interval of existence $[0,T^*)$. Then, there is an $\eps_0(m,p,H,M)$ sufficiently small such that for any $0 < \eps < \eps_0$, there are $\delta = \delta(\eps, m, p, H, M)$ and $T_0 = T_0(\eps, m, p, H, M) \in (0,1]$ such that for $h := f - \mu$, if
    $\|h_{in}\|_{L^p}^p < \delta$ then \begin{equation}\label{eq:apriori}
        \sup_{0 < t < \min(T_0,T^*)} \left(\|h(t)\|_{L^p}^p + \frac{c_0}{2}\int_{0}^{t}\int_{\R^3} \brak{v}^{-3}|\nabla h^{p/2}|^2 \;dvds \right) \le \eps,
    \end{equation}
    where $c_0 = c_0(H)$ is the constant from Lemma \ref{lem:Alower}. Moreover, $T_0$ may be chosen such that
    \begin{equation*}
    C(m,p,H,M)\eps^{1-\alpha} \le T_0,
    \end{equation*}
where $\alpha = \left(1 - \frac{1}{p} + \frac{1}{3(p-1)}\right) \in (0, 1)$.
\end{lemma}

\begin{remark}
In the proof of Lemma \ref{lem:perthame}, we study the evolution of the $L^p$-norm of $h$ via an ODE method. The $L^p$ norm satisfies a differential inequality of the form,
\begin{equation*}
y^\prime\le y + y^\alpha, \qquad \alpha \in (0,1).
\end{equation*}
For small data, the sublinear evolution dominates and restricts our estimate to a small time interval (of order $\eps^{1-\alpha}$). 
Lemma \ref{lem:perthame} provides a quantitative estimate on how \eqref{eq:landau} preserves $L^p$-smallness for a short amount of time. This is the first step towards a global-in-time existence result for small $L^p$ initial data. 
\end{remark}

\begin{proof}
We start by noting that $\mu$ is a steady state solution to \eqref{eq:landau}, so that $h$ satisfies
\begin{equation}\label{eq:landau_h}
\begin{aligned}
    \partial_t h 
        &= \nabla \cdot \left(A[f]\nabla h - \nabla a[f] h\right) + A[h]:\nabla^2\mu + h\mu.
\end{aligned}
\end{equation}
Testing \eqref{eq:landau_h} with $h^{p-1}$, integrating by parts, and rearranging yields
\begin{equation*}
\begin{aligned}
\frac{1}{p}\frac{\dd\|h(t)\|_{L^p}^p}{\dd t} &+ \frac{4(p-1)}{p^2}\int_{\R^3} A[f]\nabla h^{p/2}\cdot \nabla h^{p/2} \dd v\\
    &= \frac{p-1}{p}\int_{\R^3} h^{p+1}\dd v + \frac{2p-1}{p}\int_{\R^3} h^{p}\mu \dd v +\int_{\R^3} h^{p-1}\left(A[h]:\nabla^2\mu\right)\dd v\\
    &= I_1 + I_2 + I_3.
\end{aligned}
\end{equation*}
Beginning with $I_1$, we use H\"older and a weighted Sobolev inequality to get:
\begin{equation*}
\begin{aligned}
\int_{\R^3} h^{p+1} \dd v &= \int_{\R^3} \left(\frac{h^{p/2}}{\brak{v}^{3/2}}\right)^2 \brak{v}^3 h \dd v \le \left(\int_{\R^3} \left(h^{p/2}\right)^{6}\brak{v}^{-9} \dd v \right)^{1/3}\left(\int_{\R^3} h^{3/2}\brak{v}^{9/2} \dd v\right)^{2/3}\\
    &\le \left[C_1\int_{\R^3} \abs{\nabla h^{p/2}}^2\brak{v}^{-3} \dd v + C_2\left(\int_{\R^3} h^{3p/4} \dd v\right)^{4/3} \right]\left(\int_{\R^3} h^{3/2}\brak{v}^{9/2} \dd v\right)^{2/3},
\end{aligned}
\end{equation*}
using Lemma \ref{lem:poincare} with $s = 3/2$. H\"older's inequality and that $p>3/2$ yield:
\begin{align*}
    \left(\int_{\R^3} h^{3/2}\brak{v}^{9/2}\dd v\right)^{2/3} &\le  \left(\int_{\R^3} h^p \dd v\right)^{\frac{1}{3(p-1)}}\left(\int_{\R^3} h \brak{v}^{\frac{9}{2}\frac{p-1}{p-3/2}}\dd v\right)^{\frac{2}{3}\frac{p-3/2}{p-1}}\\
    \left(\int_{\R^3} h^{3p/4}\dd v\right)^{4/3} &\le  \left(\int_{\R^3} h^p \dd v\right)^{\frac{p - 4/3}{(p-1)}}\left(\int_{\R^3} h \dd v\right)^{\frac{p}{3(p-1)}}
\end{align*}
Summarizing, the term $I_1$ is bounded as
\begin{equation*}
I_1 \le C_1\norm{\brak{v}^{-3/2}\nabla h^{p/2}}_{L^2}^{2}\|h\|_{L^p}^{\frac{p}{3(p-1)}}\|h\brak{v}^m\|_{L^1}^{\frac{2}{3}\frac{p -\frac{3}{2}}{p - 1}} + C_2 \|h\|_{L^p}^p\|h\brak{v}^m\|_{L^1}.    
\end{equation*}
The term $I_2$ is bounded via
\begin{equation*}
    I_2 \le 2\norm{h}_{L^p}^p\|\mu\|_{L^\infty}.
\end{equation*}
The term $I_3$ is bound using H\"older's inequality, 
Lemma \ref{lem:Alower}, and conservation of mass, as
\begin{equation*}
    I_3 \le \|A[h]\|_{L^\infty}\|h\|_{L^p}^{p-1}\|\mu\|_{W^{2,p}} \leq\|\mu\|_{W^{2,p}}\|h\|_{L^p}^{p-1}\|h\|_{L^1}^{1 - \frac{p}{3(p-1)}}\|h\|_{L^p}^{\frac{p}{3(p-1)}}\le C(p)\|h\|_{L^p}^{p\left(1 - \frac{1}{p} + \frac{1}{3(p-1)}\right)}.
\end{equation*}
In summary, we have shown
\begin{equation*}
\begin{aligned}
&\frac{\dd}{\dd t} \|h(t)\|_{L^p}^p + \int_{\R^3} A[f]\nabla h^{p/2}\cdot \nabla h^{p/2} \;\dd v \\
&\quad\le C(p) \left(\left(\int_{\R^3} \brak{v}^{-3}\abs{\nabla h^{p/2}}^2 \dd v \right)\|h\|_{L^p}^{\frac{p}{3(p-1)}}\|h\brak{v}^m\|_{L^1}^{\frac{2}{3}\frac{p -\frac{3}{2}}{p - 1}} + \norm{h}_{L^p}^p\norm{h\brak{v}^m}_{L^1} + \|h\|_{L^p}^p + \|h\|_{L^p}^{p\alpha}\right),
\end{aligned}
\end{equation*}
where $\alpha = \left(1 - \frac{1}{p} + \frac{1}{3(p-1)}\right)$.
The coercivity estimate from Lemma \ref{lem:Alower} and the moment estimates from Lemma \ref{lem:moments_h} then imply that
there are constants $\overline{M} = \overline{M}(M, m, H)$ and $c_0 = c_0(H)$ for which
\begin{equation*}
\frac{c_0}{\brak{v}^3} \le A[f] \qquad \text{and} \qquad  \sup_{0 < t < \min(T^*,1)} \|h(t)\|_{L^1_m} \le \overline{M},
\end{equation*}
and thus for a fixed constant $\tilde C = \tilde C(p)$, we have
\begin{equation}\label{eq:differential_form_explicit}
\begin{aligned}
\frac{\dd}{\dd t}\|h\|_{L^p}^p
&\le -\left(\int_{\R^3}|\nabla h^{p/2}|^2\brak{v}^{-3} \;\dd v\right)\left(c_0 - \tilde{C}\|h\|_{L^p}^{\frac{p}{3(p-1)}}\overline{M}^{\frac{2}{3}\frac{p -\frac{3}{2}}{p - 1}}\right) + \tilde C(\overline{M} + 1)\|h\|_{L^p}^p + \tilde C\|h\|_{L^p}^{p\alpha},
\end{aligned}
\end{equation}
for $0 \le t < \min(1,T^*)$. Introducing $y$, $G$ and $N$ as
\begin{equation*}
y(t) := \|h(t)\|_{L^p}^p \quad \text{and} \quad G(t) := \left(\int_{\R^3}|\nabla h^{p/2}|^2\brak{v}^{-3} \;\dd v\right) \quad \text{and} \quad N(y) := \left(c_0 - \tilde {C} y^{\frac{1}{3(p-1)}}\overline{M}^{\frac{2}{3}\frac{p -\frac{3}{2}}{p - 1}}\right),
\end{equation*}
we have shown in \eqref{eq:differential_form_explicit} that $y$ satisfies the differential inequality
\begin{equation}\label{eq:differential_form}
\begin{aligned}
\frac{\dd y}{\dd t} \le -G(t)N(y)  + \tilde{C}(\overline{M} + 1)y + \tilde{C}y^\alpha.
\end{aligned}
\end{equation}
We prefer to study \eqref{eq:differential_form} in integral form: Integrating from $0$ to $t$, we find for any $0 < t < \min(T^*,1)$,
\begin{equation}\label{eq:integral_form}
y(t) \le y_0 + \tilde{C}\int_0^t (\overline{M} + 1)y(s) + y(s)^\alpha\;ds - \int_0^t G(s)N(y(s)) \;ds.
\end{equation}
We are now ready to fix our choice of parameters $\eps_0$, $T_0$, and $\delta$. First, we choose $\eps_0$ as
\begin{equation*}
\eps_0 = \left(\frac{c_0^{3(p-1)}}{(2\tilde C)^{3(p-1)}\overline{M}^{2p-3}}\right), \qquad \text{so that} \qquad N(y) > \frac{c_0}{2} \quad \text{ if and only if } \quad y < \eps_0.
\end{equation*}
Second, for any $\eps \in (0,\eps_0)$, we pick $T_0$ and $\delta$ as
\begin{equation}\label{eq:constraint}
T_0 = (3\tilde C)^{-1} \min\left(\left(\overline{M}+1\right)^{-1}, \eps^{1-\alpha}\right)  \qquad \text{and} \qquad \delta < \eps/3.
\end{equation}
The choice of $\delta$ guarantees $y_0 < \eps$ and consequently $N(y_0) > c_0/2$. 
Therefore, if $y_0 < \delta$, by continuity of $y$, either $y(t) < \eps$ for all $0 \le t \le \min(T^*,1)$ or there is a first time $t_0 \in (0,\min(T^*,1))$ such that $y(t_0) = \eps$.
By definition of $t_0$, we then have $y(t) < \eps$ and consequently $N(y(t)) > c_0/2$ for $0 \le t < t_0$. So, for any $0 < t < \min(t_0,T_0)$, we combine \eqref{eq:integral_form} with the choice of $T_0$ in \eqref{eq:constraint} and the non-negativity of $G$ to find
\begin{equation}\label{eq:diff_bound}
y(t) < \eps/3 + \tilde{C}(\overline{M} + 1)\eps t + \tilde{C}\eps ^\alpha t - \frac{c_0}{2}\int_0^t G(s) \;dt < \eps.
\end{equation}
We find $t_0 \ge T_0$ and rearranging \eqref{eq:diff_bound} yields the desired estimate
\begin{equation*}
\sup_{0 < t < \min(T_0,T^*)} y(t) + \frac{c_0}{2} \int_0^{\min(T_0,T^*)} G(s) \;\dd t < \eps.
\end{equation*}
The explicit choice of $T_0$ in \eqref{eq:constraint} yields the claimed asymptotic behavior.
\end{proof}

\section{Quantitative \texorpdfstring{$L^\infty$}{L infty} Regularization}\label{sec:DeGiorgi}
In this section, we prove a quantitative regularization estimate for the $L^\infty$ norm of smooth solutions to \eqref{eq:landau}. More precisely, we show that for smooth solutions to \eqref{eq:landau}, the $L^\infty$ norm is controlled by the energy functional that appears in \eqref{eq:apriori}. Our ultimate goal is to construct solutions to \eqref{eq:landau} for initial data that are close to equilibrium in $L^p$, even for rough profiles. Hence, we ensure that the estimate depends only on $L^1$ moments and entropy, making the estimate independent of the solution's smoothness. The proof, inspired by \cite{ABDL}, is a modification of the iteration procedure initially introduced by De Giorgi in \cite{DeGiorgi} for the study of elliptic equations. A similar version of our estimates has already been obtained by Silvestre in \cite{Silvestre}. One could possibly replace the forthcoming estimates with Silvestre's estimates to obtain a similar result to Theorem \ref{thm:existence}, where smallness is instead required in a weighted $L^p$ space. We prefer to assume smallness in the strictly weaker unweighted $L^p$ norm and so prove our own regularization estimates:

\begin{proposition}\label{prop:degiorgi}
Fix $p > \frac{3}{2}$, $m > \frac{9}{2}\frac{p-1}{p - \frac{3}{2}}$, and $H \in \R$. Suppose $f_{in}$ is a non-negative Schwartz class function, satisfying the normalization \eqref{eq:normalization} and the bound
\begin{equation*}
    \int_{\R^3} f_{in} \abs{\log(f_{in})} \;dv \le H.
\end{equation*}
Let $f:[0,T^*)\rightarrow \R^+$ be the unique Schwartz class solution to \eqref{eq:landau} with initial datum $f_{in}$ and $T^* > 0$ its maximal time of existence.
Then, for $\gamma$, $\beta_0$, $\beta_1$ and $\beta_2$, defined in terms of $m$ and $p$ via
\begin{equation}\label{eq:exponents}
\beta_0 = \frac{1}{3(p-1)}, \qquad \beta_1 = \frac{2}{3} - \frac{3}{m}, \qquad \beta_2 = \frac{3}{m}, \qquad \text{and}\qquad \gamma = \frac{2p -3}{3m}\left[m - \frac{9(p-1)}{2p-3}\right],
\end{equation}
there is a constant $C = C(p,m,H)$ such that for any $0 < t < \min(1, T^*)$ there holds
\beqs
\norm{h(t)}_{L^\infty(\R^3)} \leq C \left(E_0^{\frac{\beta_1}{\gamma}} \mathcal M^{\frac{\beta_2}{\gamma}} + E_0^{\frac{\beta_1}{1 + \gamma}} \mathcal M^{\frac{\beta_2}{1+\gamma}} + E_0^{\frac{\beta_1}{2 + \gamma}}  \mathcal M^{\frac{\beta_2}{2+\gamma}}+ E_0^{\frac{\beta_0 + \beta_1}{2 + \gamma}} \mathcal M^{\frac{\beta_2}{2+\gamma}} + E_0^{\frac{\beta_1}{1 + \gamma}}t^{-\frac{1}{1 + \gamma}} \mathcal M^{\frac{\beta_2}{1+\gamma}}\right),
\eeqs
where $\mathcal M$ and $E_0$ are given by
\begin{equation*}
\mathcal{M} := \sup_{0 < t < \min(1, T^*)} \|h(t)\|_{L^1_m}, \qquad E_0 := \sup_{0 < t < \min(1, T^*)} \norm{h(t)}_{L^p}^p + \int_{0}^{T^*}\int_{\R^3} \langle v\rangle^{-3}\abs{\nabla h^{p/2}}^2\dd v \dd t.
\end{equation*}
\end{proposition}

The rest of this section is devoted to the proof of Proposition \ref{prop:degiorgi}, which consists of three steps. In Step 1 and Step 2 we derive an energy estimate relating the energy of different level sets. In Step 3, we iterate this inequality to conclude the proof.

\begin{flushleft}
{\bf \underline{Step 1: Energy estimate for level set functions (Differential Form)}}
\end{flushleft}

For any fixed $\ell \in \R^+$, we consider the part of $h$ above $\ell$, denoted by 
\bals
	h_\ell(t, v) := h(t, v) - \ell, \quad h_\ell^+(t, v) = \max\left(h_\ell(t, v), 0\right).
\eals
Our first goal is to derive an energy estimate on these level set functions, which is contained in the following lemma:
\begin{lemma}\label{lem:enestim_diff}
Suppose $m$, $p$, $M$, $H$, and $h$ are as in Proposition \ref{prop:degiorgi}. Then, there is a constant $C = C(p,m)$ such that for any $0 \leq k < \ell$
\begin{equation}\label{eq:enestim_diff}
\begin{aligned}
    \frac{\dd}{\dd t}\int_{\R^3} (h_\ell^+)^p \dd v &+ c_0\int_{\R^3}\abs{\nabla (h_\ell^+)^{\frac{p}{2}}}^2\langle v\rangle^{-3} \dd v \\
    &\leq  C\left[\frac{1}{(\ell-k)^\gamma} + \frac{1 + \ell}{(\ell-k)^{1+\gamma}} + \frac{1 + \ell + E^{\beta_0}(t) + \ell^2}{(\ell-k)^{2 + \gamma}}\right]\\
    &\qquad\times \left[\left(\norm{\langle \cdot \rangle^{-\frac{3}{2}}\nabla (h_k^+)^{\frac{p}{2}}}_{L^{2}}^{2} +\norm{h_k^+}_{L^p}^p\right) \norm{h_k^+}_{L^p}^{p\left(\frac{q - p -\frac{2}{3}}{p-1}\right)}\norm{h_k^+}_{L^1_m}^{\frac{3}{2m}}\right],
\end{aligned}
\end{equation}
where $c_0 = c_0(M,H)$ is the constant from Lemma \ref{lem:Alower} and
\begin{equation*}
    q = \frac{5}{3}p - \frac{3(p-1)}{m}, \qquad \gamma = q - (p + 1),\qquad \beta_0 = \frac{1}{3(p-1)}, \qquad E(t) = \|h(t)\|_{L^p}^p.
\end{equation*}
\end{lemma}
\begin{proof}[Proof of Lemma \ref{lem:enestim_diff}]
We prove Lemma \ref{lem:enestim_diff} in three steps, Step 1-i to Step 1-iii. In Step 1-i we derive an energy estimate for every level set function solving \eqref{eq:landau_h}. In Step 1-ii we combine this energy estimate with a relation between different level set functions. In Step 3-iii we assemble these estimates and conclude the proof of Lemma 4.3.

\begin{flushleft}
\textbf{\underline{
\textbf{Step 1-i: Energy estimate}}}
\end{flushleft}

We test the weak formulation of \eqref{eq:landau_h} with $(h_\ell^+)^{p-1}$. Using that $\partial_t h_\ell^+ = \partial_t h \chi_{\{h \geq \ell\}}$ and $\nabla_v h_\ell^+ = \nabla_v h \chi_{\{h \geq \ell\}}$, we rearrange terms as in the proof of Lemma \ref{lem:perthame} to obtain
\bal\label{eq:enestim_aux}
	\frac{1}{p}\frac{\dd}{\dd t}&\int_{\R^3} (h_\ell^+)^p \dd v + \frac{4(p-1)}{p^2}\int_{\R^3} A[h]\abs{\nabla (h_\ell^+)^\frac{p}{2}}^2 \\
	&\leq (p-1)\int_{\R^3} \nabla h_\ell^+ h \nabla a[f] (h_\ell^+)^{p-2} \dd v + \int_{\R^3} (h_\ell^+)^{p-1}\left(A[h]:\nabla^2 \mu \right)\dd v + \int_{\R^3} h\mu (h_\ell^+)^{p-1}\dd v.
\eal
For the first term on the right hand side, we integrate by parts and continue adding and subtracting to remove the $h$ and $f$, until we have more terms, each involving only $h_\ell^+$:
\begin{align*}
    \int_{\R^3} \nabla h_\ell^+ h \nabla a[f] (h_\ell^+)^{p-2} \dd v &= \int_{\R^3} \nabla h_\ell^+  \nabla a[f] (h_\ell^+)^{p-1} \dd v+ \ell\int_{\R^3} \nabla h_\ell^+  \nabla a[f] (h_\ell^+)^{p-2} \dd v\\
    &= \frac{1}{p}\int_{\R^3} f (h_\ell^+)^{p} \dd v + \frac{\ell}{p-1} \int_{\R^3}  f (h_\ell^+)^{p-1} \dd \\
    &= \frac{1}{p}\int_{\R^3} (h_\ell^+)^{p+1} \dd v + \frac{\ell}{p-1} \int_{\R^3} (h_\ell^+)^{p} \dd v\\
    &\qquad + \frac{1}{p}\int_{\R^3} (\mu + \ell) (h_\ell^+)^{p} \dd v + \frac{\ell}{p-1} \int_{\R^3} (\mu + \ell)(h_\ell^+)^{p-1} \dd v \\
    &\leq C(p)\int_{\R^3} (h_\ell^+)^{p+1} \dd v + C(p)\left(\ell + 1\right)\int_{\R^3} (h_\ell^+)^{p} \dd v + C(p)\left(\ell+\ell^2\right) \int_{\R^3}(h_\ell^+)^{p-1} \dd v.
\end{align*}
Bounding the remaining (linear) terms on the right hand side of \eqref{eq:enestim_aux}, we find
\bals
    \int_{\R^3} (h_\ell^+)^{p-1}\left(A[h]:\nabla^2 \mu \right)\dd v &+ \int_{\R^3} h\mu (h_\ell^+)^{p-1}\dd v \\
    &\leq C\norm{A[h]}_{L^\infty}\int  (h_\ell^+)^{p-1}\dd v + \int (h_\ell^+)^{p} \dd v + \ell \int (h_\ell^+)^{p-1} \dd v\\
    &\leq \big(1 + CE(t)^{\beta_0}\big)\int  (h_\ell^+)^{p-1}\dd v + \int (h_\ell^+)^{p} \dd v + \ell \int (h_\ell^+)^{p-1} \dd v,
\eals
where $E(t) = \|h(t)\|_{L^p}^p$ and $\beta_0 = \frac{1}{3(p - 1)}$ come from the upper bound for $A$ in Lemma \ref{lem:Alower} and the constant $C$ depends only on $p$ as the mass of $f$ is fixed.

Due to the lower bound on $A[f]$ in Lemma \ref{lem:Alower}, we conclude for some universal constant $C$ depending only on $p$,
\begin{equation}\label{eq:enestim}
\begin{aligned}
    \frac{\dd}{\dd t}\int_{\R^3} (h_\ell^+)^p \dd v &+ c_0\int_{\R^3}\abs{\nabla (h_\ell^+)^{\frac{p}{2}}}^2\langle v\rangle^{-3} \dd v\\
	&\leq  C\int_{\R^3} (h_\ell^+)^{p+1}\dd v + C (1+\ell) \int_{\R^3} (h_\ell^+)^{p} \dd v + C\left(1+ E^{\beta_0} + \ell + \ell^2\right) \int_{\R^3}(h_\ell^+)^{p-1} \dd v.
\end{aligned}
\end{equation}
After obtaining an energy estimate like \eqref{eq:enestim}, the standard way of applying the De Giorgi approach to parabolic equations is to first integrate over time and use a Sobolev embedding in the variable $v$ to obtain control of the norms $L^\infty_t L^p_v$ and $L^p_tL^{p^*}_v$, where $p^*$ is the Sobolev conjugate of $p$. Next, one interpolates to gain control of $L^q_{t,v}$ for some $q > p$. In the case of the heat equation, this results in $q = 10/3$, providing improved integrability. The De Giorgi method then repeats this process iteratively, ultimately concluding a quantitative $L^\infty_{t,v}$ bound. However, the presence of the weight $\brak{v}^{-3}$ in \eqref{eq:enestim} prevents the direct application of a Sobolev inequality, but, similar to \cite{ABDL}, a modified De Giorgi procedure can still be used to gain integrability.

\begin{flushleft}
\textbf{\underline{
\textbf{Step 1-ii: Gain of integrability}}}
\end{flushleft}

By passing from a lower level set to a higher lever set, we observe that for any $0 \leq k < \ell$,
\beqs
	0 \leq h_\ell^+ \leq h_k^+,
\eeqs
and if $h_\ell \geq 0$, then $h \geq \ell > k$. In this case, $h_k = h_\ell + \ell - k$ and $\frac{h_k}{\ell-k} = \frac{h_\ell}{\ell-k} + 1 \geq 1$. Therefore,
\beqs
	 \mathbbm{1}_{\{h_\ell \geq 0\}} \leq \frac{h_k^+}{\ell-k}.
\eeqs
In particular for any $\alpha > 0$
\beqs
	 \mathbbm{1}_{\{h_\ell \geq 0\}} \leq \left(\frac{h_k^+}{\ell-k}\right)^\alpha.
\eeqs
We deduce
\beq
	h_\ell^+ \leq (\ell-k)^{-\alpha} (h_k^+)^{1+\alpha} \quad \text{for any }\alpha > 0.
\label{eq:flfk}
\eeq
Inequality (\ref{eq:flfk}) is crucial for the proof of the next lemma, which explains how to estimate the terms on the right hand side of \eqref{eq:enestim}.
\begin{lemma}
\label{lem:interpolate}
Let $p$, $m$, and $h$ be as in Proposition \ref{prop:degiorgi}. Then, for each $0 \le k < \ell$ and for $i \in \set{0, 1, 2}$ there holds
\begin{equation*}
    \int_{\R^3} (h_\ell^+)^{p + 1 - i} \le \frac{C}{(\ell - k)^{i + \gamma}}\left[\left(\norm{\langle \cdot \rangle^{-\frac{3}{2}}\nabla (h_k^+)^{\frac{p}{2}}}_{L^{2}}^{2} +\norm{h_k^+}_{L^p}^p\right) \norm{h_k^+}_{L^p}^{p\left(\frac{2}{3}-\frac{3}{m}\right)}\norm{h_k^+}_{L^1_m}^{{{\frac{3}{m}}}}\right],
\end{equation*}
where
\begin{equation*}
    \gamma = \frac{2}{3}\left( p - \frac{3}{2}\right)-\frac{3}{m}(p-1).
\end{equation*}
\end{lemma}

\begin{proof}[Proof of Lemma \ref{lem:interpolate}]
First, let us derive a useful bound on $\|g\|_{L^r}^r$, where $g:\R^3 \rightarrow \R$ is arbitrary. Interpolating, we find
\beqs
    \norm{g}_{L^r} \leq\norm{\langle \cdot \rangle^{-\frac{3}{p}}g}_{L^{3p}}^{\theta_1} \norm{ g}_{L^p}^{\theta_2}\norm{\langle\cdot\rangle^\alpha g}_{L^1}^{\theta_3},
\eeqs
provided $\theta_1,\ \theta_2,\ \theta_3 \in (0,1)$ and the following relations are satisfied: 
\beqs
    \theta_1 + \theta_2 + \theta_3 = 1, \qquad \frac{\theta_1}{3p} + \frac{\theta_2}{p} + \theta_3 = \frac{1}{r}, \qquad \frac{-3}{p}\theta_1 + \alpha \theta_3 = 0.
\eeqs
Imposing further that $r\theta_1 = p$, we find
\bals
    r\theta_2 = \frac{p\left[r - p -\frac{2}{3}\right]}{p-1},\qquad r\theta_3 = \frac{5p-3r}{3(p-1)},\qquad \alpha = \frac{9(p-1)}{5p-3r}, \qquad r \in \left(p+\frac{2}{3}, \frac{5p}{3}\right).
\eals
The weighted Sobolev inequality of Lemma  \ref{lem:poincare} with $s = 2$ then implies for a constant depending only on $p$, $r$ and $\alpha$
\beq \label{eqn:interpolation_estimate}
    \norm{g}^r_{L^r} \leq C\left(\norm{\langle \cdot \rangle^{-\frac{3}{2}}\nabla \big(g^{\frac{p}{2}}\big)}_{L^{2}}^{2} +\norm{g}_{L^p}^p\right) \norm{g}_{L^p}^{p\left(\frac{r - p -\frac{2}{3}}{p-1}\right)}\norm{\langle\cdot\rangle^\alpha g}_{L^1}^{\frac{5p-3r}{3(p-1)}}.
\eeq
Note that the admissible range for $r$, i.e. $(p+ 2/3, 5p/3)$ is non-degenerate if and only if $p > 1$ and, moreover, $p+1$ is within this range if and only if $p > 3/2$.
With the estimate \eqref{eqn:interpolation_estimate} in hand for general $g$ and $r$, let us return to our specific setting.
For simplicity, let us only prove the case when $i = 0$.
Let us fix $q$ by setting $\alpha = m$, or, in other words, define $q$ via the relation
\begin{equation*}
m = \frac{9(p-1)}{5p - 3q}  \qquad \text{or, equivalently,} \qquad q = \frac{5p}{3} - \frac{3(p-1)}{m}.
\end{equation*}
With this definition of $q$, we see that $q$ is in the admissible range $(p+2/3, 5p/3)$ provided $m > 9/2$. However, since we need to bound the $L^{p+1}$ norm, we see that $q > p + 1$ if and only if
\begin{equation*}
m > \frac{9}{2} \frac{p- 1}{p - \frac{3}{2}}.
\end{equation*}
Since by assumption we have enough moments, we set $\gamma = q - (p + 1)$ for $\gamma$ positive.
Finally, we combine \eqref{eq:flfk} with \eqref{eqn:interpolation_estimate} to obtain for $0 \le k < \ell$,
\begin{align*}
\int_{\R^3} (h_\ell^+)^{p+1} \dd v &\leq \frac{1}{(\ell-k)^{\gamma}}\int (h_k^+)^{p+1+\gamma} \dd v\\
    &\leq \frac{1}{(\ell-k)^{\gamma}}\int (h_k^+)^{q} \dd v\\
    &\leq \frac{C}{(\ell-k)^{\gamma}}\left[\left(\norm{\langle \cdot \rangle^{-\frac{3}{2}}\nabla (h_k^+)^{\frac{p}{2}}}_{L^{2}}^{2} +\norm{h_k^+}_{L^p}^p\right) \norm{h_k^+}_{L^p}^{p\left(\frac{q - p -\frac{2}{3}}{p-1}\right)}\norm{h_k^+}_{L^1_m}^{{{\frac{5p-3q}{3(p-1)}}}}\right],\\
\end{align*}
where both $q$ and $\gamma$ have been fixed in terms of $m$ and $p$.
\end{proof}

\begin{flushleft}
\textbf{\underline{Step 1-iii: Conclusion of energy estimate on level sets}}
\end{flushleft}

Lemma \ref{lem:interpolate} bounds the terms on the right hand side of \eqref{eq:enestim} using only the quantities appearing on the left hand side, with the exception of $L^1$-moments, which we know stay bounded by Lemma \ref{lem:moments_h}. This yields the final differential form of our energy estimate:
\begin{equation*}
\begin{aligned}
    \frac{\dd}{\dd t}\int_{\R^3} (h_\ell^+)^p \dd v &+ c_0\int_{\R^3}\abs{\nabla (h_\ell^+)^{\frac{p}{2}}}^2\langle v\rangle^{-3} \dd v \\
    &\leq  C\left[\frac{1}{(\ell-k)^\gamma} + \frac{1 + \ell}{(\ell-k)^{1+\gamma}} + \frac{1 + E^{\beta_0} + \ell + \ell^2}{(\ell-k)^{2 + \gamma}}\right]\\
    &\qquad\times \left[\left(\norm{\langle \cdot \rangle^{-\frac{3}{2}}\nabla (h_k^+)^{\frac{p}{2}}}_{L^{2}}^{2} +\norm{h_k^+}_{L^p}^p\right) \norm{h_k^+}_{L^p}^{p\left(\frac{q - p -\frac{2}{3}}{p-1}\right)}\norm{h_k^+}_{L^1_m}^{\frac{5p-3q}{3(p-1)}}\right],
\end{aligned}
\end{equation*}
and concludes the proof of Lemma \ref{lem:enestim_diff}.
\end{proof}

\begin{flushleft}
    {\bf \underline{Step 2: Energy estimate for level set functions (Integral Form)}}
\end{flushleft}
The energy estimate \eqref{eq:enestim_diff} gives us control between different level sets, however, we will find this more useful in  integral form. This leads to the definition of our energy functional $\mathcal{E}$ for any $\ell \geq 0$ and $0 \le T_1 \le T_2 \le T^*$ as
\beqs
	\mathcal E_\ell(T_1, T_2) := \sup_{t \in [T_1, T_2]} \norm{h_\ell^+(t)}_{L^p_v}^p + c_0\int_{T_1}^{T_2} \norm{\langle \cdot\rangle^{-\frac{3}{2}}\nabla (h_\ell^+)^{\frac{p}{2}}}^2_{L^2_v} \dd t,
\eeqs
where $c_0$ is the constant from Lemma \ref{lem:Alower}.
We will now integrate \eqref{eq:enestim_diff} to obtain an estimate purely in terms of this energy:
\begin{lemma}\label{lem:ElEk}
Let $m$, $p$, $H$, $\mathcal M$ and $h$ be as in Proposition \ref{prop:degiorgi}. Further, let $\beta_0$, $\gamma$, and $q$ be the exponents from Lemma \ref{lem:enestim_diff}. Then, there is a constant $C = C(m, p, H)$ such that for any $0 \leq T_1 \leq T_2 \leq T_3 \leq T^*$ and $0 \leq k \leq \ell$, 
there holds
\begin{equation*}
\begin{aligned}
    \mathcal{E}_\ell(T_2,T_3) &\le C(T_3-T_1)\mathcal M^{ \beta_2}\mathcal{E}_k(T_1,T_3)^{1+\beta_1}\\
    &\qquad\times\left[\frac{1}{(T_2 - T_1)(\ell - k)^{1 + \gamma}} + \frac{1}{(\ell - k)^\gamma} + \frac{1 + \ell}{(\ell - k)^{1 + \gamma}} + \frac{1 + \ell +\ell^2 + \mathcal{E}_0^{\beta_0}(T_1, T_3)}{(\ell - k)^{2+\gamma}} \right],
\end{aligned}
\end{equation*}
where
\begin{equation*}
    \beta_1 = \frac{2}{3}-\frac{3}{m} \qquad \text{and} \qquad \beta_2 = \frac{3}{m}.
\end{equation*}
\end{lemma}

\begin{proof}
Fix $0 \leq T_1 < T_2 \leq T_3 \leq T^*$. 
We integrate \eqref{eq:enestim_diff} over $(t_1, t_2)$ and denoting the right hand side of \eqref{eq:enestim_diff} by $RHS(t)$, we obtain
\begin{equation*}
\int_{\R^3} (h_\ell^+)^p(t_2) \dd v + c_0 \int_{t_1}^{t_2}\int_{\R^3} \abs{\nabla (h_\ell^+)^{\frac{p}{2}}}^2\brak{v}^{-3} \dd v \dd t \le \int_{\R^3} (h_\ell^+)^p(t_1) \dd v + \int_{t_1}^{t_2} RHS(t) \dd t.
\end{equation*}
Since $RHS(t)$ is positive, this implies for $T_1 \leq t_1 \leq T_2 \leq t_2 \leq T_3$,
\bals
    \int_{\R^3} (h_\ell^+)^p(t_2) \dd v+ c_0 \int_{T_2}^{t_2} \int_{\R^3} \abs{\nabla (h_\ell^+)^{\frac{p}{2}}}^2 \langle v\rangle^{-3} \dd v\dd t \le \int_{\R^3} (h_\ell^+)^p(t_1) \dd v + \int_{T_1}^{t_2} RHS(t) \dd t.
\eals
Next, taking the the supremum over $t_2 \in [T_2, T_3]$ and then averaging over $t_1 \in [T_1, T_2]$, we find
\bals
    \sup_{t\in [T_2, T_3]} \int_{\R^3} (h_\ell^+)^p(t) \dd v &+ c_0 \int_{T_2}^{T_3} \int_{\R^3} \abs{\nabla (h_\ell^+)^{\frac{p}{2}}}^2 \langle v\rangle^{-3} \dd v\dd t\\
    &\le \frac{1}{T_2 - T_1}\int_{T_1}^{T_2}\int_{\R^3} (h_\ell^+)^p(t) \dd v \dd t + \int_{T_1}^{T_3} RHS(t) \dd t.
\eals
The left hand side is $\mathcal{E}_\ell(T_2,T_3)$. It remains to relate the right hand side back to the energy functional $\mathcal{E}_k(T_1,T_3)$. Indeed, we have the obvious estimate
\bals
\int_{T_1}^{T_3}\Bigg[\left(\norm{\langle \cdot \rangle^{-\frac{3}{2}}\nabla (h_k^+)^{\frac{p}{2}}}_{L^{2}}^{2} +\norm{h_k^+}_{L^p}^p\right) &\norm{h_k^+}_{L^p}^{p\left( \frac{2}{3}-\frac{3}{m}\right)}\norm{h_k^+}_{L^1_m}^{{\frac{3}{m}}}\Bigg] \dd t  \\
&\leq\mathcal{M}^{\beta_2}\left(c_0^{-1} + (T_3 - T_1)\right)\mathcal{E}_k(T_1,T_3)^{1 + \beta_1},
\eals
where
\beqs
\beta_1 =  \frac{2}{3}-\frac{3}{m}, \qquad \beta_2 = \frac{3}{m}, \qquad \text{and} \qquad \mathcal{M} = \sup_{T_1 < t < T_3} \|h(t)\|_{L^1_m}.
\eeqs
Also, recalling the definition of $E(t)$, we see
\begin{equation*}
\sup_{T_1 < t < T_3} E(t)^{\beta_0} = \sup_{T_1 < t < T_3} \|h(t)\|_{L^p}^{p\beta_0} \le \mathcal{E}_0(T_1, T_3)^{\beta_0}.
\end{equation*}
Therefore, we see 
\bals
\int_{T_1}^{T_3} &RHS(t) \dd t \\
&\le C\mathcal{M}^{\beta_2}\left(c_0^{-1} + (T_3 - T_1)\right)\mathcal{E}_k(T_1,T_3)^{1 + \beta_1}\left[\frac{1}{(\ell - k)^\gamma} + \frac{1 + \ell}{(\ell - k)^{1 + \gamma}} + \frac{1 + \ell + \ell^2 + \mathcal{E}_0(T_1, T_3)^{\beta_0}}{(\ell - k)^{2+\gamma}} \right],
\eals
where the constant $C = C(p,m)$ is still independent of $h$, the level sets $l$ and $k$, and the times $T_1$, $T_2$, and $T_3$. Finally, using Lemma \ref{lem:interpolate} once more, we obtain
\begin{equation*}
\frac{1}{T_2 - T_1} \int_{T_1}^{T_2}\int_{\R^3} (h_\ell^+)^p \dd v \dd t \le \frac{C\mathcal{M}^{\beta_2}\left(c_0^{-1} + (T_3 - T_1)\right)\mathcal{E}_k(T_1,T_3)^{1 + \beta_1}}{(T_2 - T_1)(\ell-k)^{1 +\gamma}}.
\end{equation*}
By Lemma \ref{lem:moments_h}, $\mathcal{M}$ grows at most linearly in $T_3$, and thus we conclude
\begin{equation*}
\begin{aligned}
    \mathcal{E}_\ell(T_2,T_3) &\le C(H,p,m)
    (T_3-T_1)\mathcal M^{\beta_2}
    \mathcal{E}_k(T_1,T_3)^{1+\beta_1}\\
        &\qquad\times\left[\frac{1}{(T_2 - T_1)(\ell - k)^{1 + \gamma}} + \frac{1}{(\ell - k)^\gamma} + \frac{1 + \ell}{(\ell - k)^{1 + \gamma}} + \frac{1 +\ell + \ell^2 + \mathcal{E}_0(T_1, T_3)^{\beta_0}}{(\ell - k)^{2+\gamma}} \right],
\end{aligned}
\end{equation*}
as desired.
This concludes the proof.
\end{proof}

\begin{flushleft}
    {\bf \underline{Step 3: De Giorgi iteration}}
\end{flushleft}

\begin{proof}[Proof of Proposition \ref{prop:degiorgi}]
We now finish the proof of Proposition \ref{prop:degiorgi} using an iteration procedure. Let $0 < t < T^*$ be the fixed times in the statement of Proposition \ref{prop:degiorgi} and take $t < T \leq \min(T^*, 1)$. We consider for $n \in \N$
\beqs	
	\ell_n = K\left(1 - 2^{-n}\right), \quad t_n = t\left(1 - 2^{-n}\right), \quad \text{and} \quad E_n = \mathcal{E}_{\ell_n}(t_n, T),
\eeqs
where $K > 0$ is a parameter to be chosen appropriately. Indeed, our main goal is to find a value of $K$ for which $\lim_{n\rightarrow \infty} E_n = 0$. By the definition of the energy and the choice of iteration parameters, this would imply $h(\tau,v) \le K$ for almost every $t < \tau < T$ and $v\in\R^3$.

To this end, we apply Lemma \ref{lem:ElEk} with $T_1 = t_n$, $T_2 = t_{n+1}$, and $T_3 = T$ and deduce the following recurrence relation for $E_n$
\bal
	E_{n+1} &\leq C (T-t_n)\mathcal M^{\beta_2}E_n^{1 + \beta_1}\Bigg[ \frac{1}{(t_{n+1}-t_n)(\ell_{n+1}-\ell_n)^{\gamma+1}} + \frac{1}{(\ell_{n+1} - \ell_n)^\gamma} \\
    &\qquad\qquad\qquad\qquad\qquad+ \frac{1+\ell_{n+1}}{(\ell_{n+1}-\ell_n)^{\gamma+1}} +\frac{1 + E_0^{\beta_0} +\ell_{n+1} + \ell_{n+1}^2}{(\ell_{n+1}-\ell_n)^{\gamma+2}}\Bigg]\\
    &\leq C
     \mathcal M^{\beta_2}E_n^{1 + \beta_1}\Bigg[\frac{2^{(\gamma + 1)(n+1)}}{t K^{\gamma+1}} + \frac{2^{(\gamma + 1)(n+1)}+ 2^{(\gamma + 2)(n+1)}(1 - 2^{-(n+1)})}{K^{\gamma+1}} \\
    &\qquad\qquad\qquad\qquad\qquad+ \frac{2^{(\gamma+2)(n+1)}(1 + E_0^{\beta_0})}{K^{\gamma+2}}\\
    &\qquad\qquad\qquad\qquad\qquad + \frac{2^{\gamma(n+1)} + 2^{(\gamma + 1)(n+1)}\left(1 - 2^{-(n+1)}\right) + 2^{(\gamma + 2)(n+1)}\left(1 - 2^{-(n+1)}\right)^2}{K^\gamma}\Bigg]\\
	&\le C\mathcal M^{\beta_2}E_n^{1 + \beta_1}\left[\frac{2^{n(\gamma+1)}}{t K^{\gamma+1}} +\frac{2^{n(\gamma+2)}\left(1 + E_0^{\beta_0}\right)}{K^{\gamma+2}} + \frac{2^{n(\gamma+2)}}{K^{\gamma+1}} + \frac{2^{n(\gamma+2)}}{K^{\gamma}}\right].
\label{eq:recurrence}
\eal
Now, because the power $\beta_1$ is positive, $E_n$ should decay, even though the coefficients in the recurrence are geometric (in $n$). To prove this, we use a trick from \cite{ABDL} to find a quantitative estimate for $K$ despite differing homogeneity in $K$. In essence, we expect $E_n$ to decay exponentially, so we search for a choice of parameters $K$ and $Q > 0$ such that the sequence $E_n^*$ defined as
\beqs
    E_n^* := E_0 Q^{-n}, \quad n \in \N
\eeqs
satisfies \eqref{eq:recurrence} with the reversed inequality, and so by induction will remain larger than $E_n$. Therefore, we want
\bal
    E_{n+1}^* \geq C \mathcal M^{\beta_2}(E_n^*)^{1+\beta_1}\left[\frac{2^{n(\gamma+1)}}{tK^{\gamma+1}}+\frac{2^{n(\gamma+2)}\left(1 + E_0^{\beta_0}\right)}{K^{\gamma+2}}+\frac{2^{n(\gamma+2)}}{K^{\gamma+1}}+ \frac{2^{n(\gamma+2)}}{K^{\gamma}}\right].
\label{eq:reverse}
\eal
Using the definition of $E_n^*$, \eqref{eq:reverse} holds if
\bals
    1 \geq C \mathcal M^{\beta_2}E_0^{\beta_1}Q \left(2^{\gamma+2} Q^{-\beta_1}\right)^n\left[\frac{1}{tK^{\gamma+1}} +\frac{1 + E_0^{\beta_0}}{K^{\gamma+2}}+\frac{1}{K^{\gamma+1}}+ \frac{1}{K^\gamma}\right].
\eals
We now choose $Q$ such that 
\beqs
    2^{\gamma+2}Q^{-\beta_1} \leq 1 \qquad \text{or equivalently,}\qquad Q \geq 2^{\frac{(\gamma+2)}{\beta_1}}.
\eeqs
Then, for \eqref{eq:reverse} to hold we need
\bals
    1 \geq C \mathcal M^{\beta_2} E_0^{\beta_1}Q \left[\frac{1}{tK^{\gamma+1}} +\frac{1 + E_0^{\beta_0}}{K^{\gamma+2}}+\frac{1}{K^{\gamma+1}} + \frac{1}{K^\gamma}\right].
\eals
Next, choosing $K$ so each of the terms is smaller than a quarter, the recurrence \eqref{eq:reverse} holds. More precisely, we choose $K$ as
\begin{equation} \label{eq:max_K}
\begin{aligned}
    K =  C \max \left\{E_0^{\frac{\beta_1}{\gamma}} \mathcal M^{\frac{\beta_2}{\gamma}}; E_0^{\frac{\beta_1}{1 + \gamma}} \mathcal M^{\frac{\beta_2}{1+\gamma}}; E_0^{\frac{\beta_1}{2 + \gamma}} \mathcal M^{\frac{\beta_2}{2+\gamma}}; E_0^{\frac{\beta_0 + \beta_1}{2 + \gamma}} \mathcal M^{\frac{\beta_2}{2+\gamma}} ; E_0^{\frac{\beta_1}{1 + \gamma}}t^{-\frac{1}{1 + \gamma}} \mathcal M^{\frac{\beta_2}{1+\gamma}}\right\}.
\end{aligned}
\end{equation}
Then, as discussed above, since $E_0 = E_0^*$ it follows by induction that $E_n \leq E_n^*$ for $n\in \N$. Moreover, since $Q > 1$ we deduce
\beqs
    \lim_{n \to \infty} E_n \le \lim_{n \to \infty} E_n^* = 0.  
\eeqs
Therefore, the definition of our energy functional $E_n$ implies 
\beqs
	\sup_{t < \tau < T} \norm{h_K^+(\tau)}_{L^p_v}^p = 0, \qquad \text{or equivalently,} \qquad \|h\|_{L^\infty((t,T)\times \R^3)} \le K.
\eeqs
Finally, from the choice of $K$, we have for all $v\in \mathbb{R}^3$ and $\tau \in(t,T)$
\begin{equation*}
h(\tau,v) \le K \le C\left(E_0^{\frac{\beta_1}{\gamma}} \mathcal M^{\frac{\beta_2}{\gamma}} + E_0^{\frac{\beta_1}{1 + \gamma}}\mathcal M^{\frac{\beta_2}{1+\gamma}} + E_0^{\frac{\beta_1}{2 + \gamma}}  \mathcal M^{\frac{\beta_2}{2+\gamma}}+ E_0^{\frac{\beta_0 + \beta_1}{2 + \gamma}} \mathcal M^{\frac{\beta_2}{2+\gamma}} + E_0^{\frac{\beta_1}{1 + \gamma}}t^{-\frac{1}{1 + \gamma}} \mathcal M^{\frac{\beta_2}{1+\gamma}}\right).
\end{equation*}
Taking $T = 2t$ yields the desired estimate. Finally, recalling the definitions of the exponents $\beta_0$, $\beta_1$, $\beta_2$, and $\gamma$, and performing some basic algebraic manipulations gives the simplified expressions found in \eqref{eq:exponents} and completes the proof of Proposition \ref{prop:degiorgi}.
\end{proof}

\section{Global Existence for Smooth Initial Data}\label{sec:smooth_data}

In this section, we prove Proposition \ref{prop:smooth_data} by demonstrating global-in-time existence of solutions for initial data that are very smooth and rapidly decaying. Let us define the class of initial data considered in the proposition as:
\begin{equation*}
\begin{aligned}
\mathcal{A}(p,m,H,M,\delta) &= \Bigg\{h_{in} \in\S(\R^3) \mid f_{in} = h_{in} + \mu \text{ satisfies } \eqref{eq:normalization},\ 0 \le f_{in},\\
    &\qquad\qquad\|f_{in}\|_{L^1_m} \le M, \ \int_{\R^3} f_{in}|\log(f_{in})| \dd v \le H, \ \|h_{in}\|_{L^p}^p \le \delta\Bigg\}. 
\end{aligned}
\end{equation*}
We now fix $p$, $m$, $H$, $M$ such that $p > 3/2$ and
$$m > \frac{9}{2}\frac{p - 1}{p - \frac{3}{2}}\max\left(1, \frac{p\left(p-\frac{3}{2}\right)}{p^2 - 2p + 3/2}\right),$$ and attempt to find a corresponding $\delta$ so that for each $h_{in} \in \mathcal{A}(\delta)$, there exists a unique global-in-time Schwartz class solution to \eqref{eq:landau} with initial datum $f_{in} = h_{in} + \mu$.

By Theorem \ref{thm:henderson_snelson_tarfulea}, for any $h_{in} \in \S(\R^3)$, there is a time $T^*(h_{in}) > 0$ and function $h:[0,T^*)\times \R^3 \rightarrow \R$ such that $f(t) = h(t) + \mu$ is the unique Schwartz class solution to \eqref{eq:landau} on $[0,T^*)\times \R^3$ with initial datum $f_{in} = h_{in} + \mu$ with maximal time of existence $T^*$. Moreover, $T^*$ is characterized via the blow up criterion
\begin{equation}\label{eq:blowup_criterion}
\lim_{t\nearrow T^*} \|f(t)\|_{L^\infty} = \infty.
\end{equation}
Our first goal is to show that $T^*$ is uniformly large for all $h_{in} \in \mathcal{A}(\delta)$. By Lemma \ref{lem:perthame}, for any $\eps > 0$ sufficiently small, there are $T_0(\eps) \le 1$ and $\delta_0(\eps) > 0$ such that for any $h_{in} \in \mathcal{A}(\delta_0)$, the following estimate holds:
\begin{equation*}
\sup_{0 < t < \min(T_0,T^*)} \|h\|_{L^p}^p + c_0(H)\int_0^{\min(T_0,T^*)}\int_{\R^3}\brak{v}^{-3}\abs{\nabla h^{p/2}}^2 \dd v\dd t \le \eps.
\end{equation*}
By Proposition \ref{prop:degiorgi} and Lemma \ref{lem:moments_h}, we find for $0 < t < \min(T_0,T^*)$ and for $\eps$ sufficiently small
\begin{equation}\label{eqn:Linfty_bound}
\|h(t)\|_{L^\infty} \le C(p,m,H,\mathcal M)\left(\eps^{\min\left(\frac{\beta_1}{\gamma}, \frac{\beta_1}{1+\gamma}, \frac{\beta_1}{2 + \gamma}, \frac{\beta_0 + \beta_1}{2 + \gamma}\right)} + \left(\eps^{\beta_1}t^{-1}\right)^{\frac{1}{1+\gamma}}\right) < \infty,
\end{equation}
where the exponents $\beta_0$, $\beta_1$, and $\gamma$ are defined in Proposition \ref{prop:degiorgi}.
Using the characterization of $T^*$ in \eqref{eq:blowup_criterion}, we must have $T^* > T_0$ for every $h_{in} \in \mathcal{A}(\delta_0(\eps))$. Moreover, Lemma \ref{lem:perthame} guarantees $T_0(\eps) \ge C\eps^{1-\alpha}$ for $\alpha = 1 - \frac{1}{p} + \frac{1}{3(p-1)}$, so that \eqref{eqn:Linfty_bound} implies
\begin{equation}\label{eqn:uniform_Linfty}
\lim_{\eps \to 0^+} \sup_{h_{in} \in \mathcal{A}(\delta_0(\eps))} \left[\sup_{T_0(\eps)/4 < t < T_0(\eps)} \|h(t)\|_{L^\infty} + \int_0^{T_0(\eps)}\int_{\R^3}\brak{v}^{-3} \abs{\nabla h^{p/2}}^2 \dd v \dd t\right] = 0,
\end{equation}
provided $\alpha + \beta_1 > 1$. We note that simplifying the expressions for $\alpha$ and $\beta_1$,
\begin{equation*}
\alpha + \beta_1 - 1 = 0 \qquad\text{if and only if}\qquad m = \frac{9}{2}\frac{p - 1}{p - \frac{3}{2}}\left(\frac{p\left(p-\frac{3}{2}\right)}{p^2 - 2p + 3/2}\right). 
\end{equation*}
Therefore, our constraint on $m$ exactly guarantees that $\alpha + \beta_1 - 1 > 0$ and so \eqref{eqn:uniform_Linfty} holds.

We would like to apply the global-in-time existence result Theorem \ref{thm:desvillettes_he_jiang} to conclude that there is a $\delta$ sufficiently small such that for each $h_{in} \in \mathcal{A}(\delta)$, we have $T^*(h_{in}) = \infty$. Consequently, $f_{in}$ admits a global-in-time Schwartz class solution if $f_{in}$ is suitably close to the Maxwellian $\mu$ in $L^p$. However, to apply Theorem \ref{thm:desvillettes_he_jiang} to the profile $h(T_0/2)$, which is well-defined for $h_{in} \in \mathcal{A}(\delta_0(\eps))$, we must show the condition:
\begin{equation}\label{eq:existence_condition}
\lim_{\eps \to 0^+} \sup_{h_{in} \in \mathcal{A}(\delta_0(\eps))} \left(\|h(T_0/2)\|_{L^2_1} + \|\nabla h(T_0/2)\|_{L^2_2}\right) < \eps_1,
\end{equation}
where $\eps_1$ is the fixed smallness parameter in Theorem \ref{thm:desvillettes_he_jiang}. Supposing for the moment that \eqref{eq:existence_condition} holds, we will show how to deduce Proposition \ref{prop:smooth_data}. 

Use \eqref{eq:existence_condition} to pick $\eps_0$ sufficiently small so that 
\begin{equation*}
\sup_{h_{in} \in \mathcal{A}(\delta_0(\eps_0))} \left(\|h(T_0/2)\|_{L^2_1} + \|\nabla h(T_0/2)\|_{L^2_2}\right) < \eps_1.
\end{equation*}
Then, for each $h_{in} \in \mathcal{A}(\delta_0(\eps_0))$, it follows that one may apply Theorem \ref{thm:desvillettes_he_jiang} to $h(T_0/2)$ and obtain a solution $\tilde f = \tilde h + \mu$ to \eqref{eq:landau} on $[T_0/2, \infty)$, where $\tilde f(T_0/2) = f(T_0/2)$, which further satisfies $\sup_{t > T_0/2}\|\tilde f\|_{L^2} < \infty$. Note that restricting $f$ to $[T_0/2, T_0]$ yields another solution to \eqref{eq:landau} with the same profile at time $T_0/2$. Since both $f$ and $\tilde f$ lie within $L^1(T_0/2, T_0; L^\infty)$ by Proposition \ref{prop:degiorgi}, applying Theorem \ref{thm:Fournier}, we conclude $f$ and $\tilde f$ agree on $[T_0/2, T_0]$. Therefore,
\begin{equation*}
f(t) = \begin{cases} f(t) &\text{if }t \le T_0/2\\
    \tilde f(t) &\text{if }t \ge T_0/2,\end{cases}
\end{equation*}
defines a global-in-time solution to \eqref{eq:landau} with initial datum $f_{in}$.
Finally, since $\tilde f$ is constructed by Theorem \ref{thm:desvillettes_he_jiang}, we obtain the long-time bound
\begin{equation*}
\sup_{T_0/2 < t < \infty} \|h(t)\|_{H^1} \le C(p, m, H, M).
\end{equation*}
Applying Proposition \ref{prop:degiorgi} once more yields that $\|h(t)\|_{L^\infty}$ remains finite, and by \eqref{eq:blowup_criterion} we conclude $T^* = \infty$. The constructed solution $f$ remains Schwartz class, which concludes the proof of Proposition \ref{prop:smooth_data}.

\noindent{\underline {\bf Proof of Claim \ref{eq:existence_condition}}}

We now show condition \eqref{eq:existence_condition}.
The first term in \eqref{eq:existence_condition} is easily controlled by using the moment bound from Lemma \ref{lem:moments_h} and \eqref{eqn:uniform_Linfty}. Indeed, for any $h_{in} \in \mathcal{A}(\delta_0(\eps))$
\begin{equation*}
\int_{\R^3} |h(T_0/2)|^2 \brak{v}^2 \dd v \le \left(\sup_{T_0/4 < t < T_0} \|h(t)\|_{L^\infty}\right)\left(\sup_{0<t<T_0} \|h(t)\|_{L^1_2}\right),
\end{equation*}
where the right hand side decays as $\eps \rightarrow 0^+$.

The second term in \eqref{eq:existence_condition} requires more care. We know by \eqref{eqn:uniform_Linfty} the $L^\infty$ norm of $h(T_0/2)$ is uniformly small for $h_{in}\in \mathcal{A}(\delta_0(\eps))$. We need to show that this implies the gradient $\nabla h(T_0/2)$ is also small (uniformly in $h_{in}$). Note that because $T_0(\eps) \sim \eps^{1-\alpha}$ for $\alpha \in (0,1)$, the $H^1$ estimates in the forthcoming Lemma \ref{lem:higher_regularity} imply for $m \ge 8$,
\begin{equation*}
\text{For }h_{in} \in \mathcal{A}(\delta_0(\eps)), \qquad \|\nabla h(T_0/2)\|_{L^2_2} \le C(H,M) \eps^{\alpha},
\end{equation*}
which concludes the proof of \eqref{eq:existence_condition}. Now, let us prove the necessary higher regularity estimates:
\vspace{-.25cm}
\begin{lemma}[Uniform $H^1$ Regularity]\label{lem:higher_regularity}
Let $f$ be a smooth, rapidly decaying solution to \eqref{eq:landau} on $[0,T] \times \R^3$ for some $0 < T < 1$, satisfying the normalization \eqref{eq:normalization} and with entropy $H$. Suppose further that $h = f - \mu$ satisfies $\|h\|_{L^\infty([0,T]\times \R^3)} < \eps \le 1$. Then, for any $k \ge 0$, we have the following weighted $H^1$ estimates: For any $0 < t < T$,
\begin{equation*}
    \sup_{t < s < T} \int_{\R^3} \brak{v}^k \abs{\nabla h(s)}^2 \dd v + c_0\int_{t}^T \brak{v}^{k-3}\abs{\nabla^2 h(s)}^2 \dd v\dd s \le C\left(k,H,\|h_{in}\|_{L^1_{k+6}}\right)\eps \left(1 + \frac{1}{t}\right).
\end{equation*}
\end{lemma}
\begin{proof}
As a first step towards $H^1$ regularity, we perform a weighted $L^2$ estimate by testing \eqref{eq:landau_h} with $\brak{v}^k h$ and integrating by parts.
\begin{equation*}
\begin{aligned}
    \frac{\dd}{\dd t}\int_{\R^3} \brak{v}^k h^2 + \int_{\R^3} \brak{v}^kA[f] \nabla h \cdot \nabla h &\le -\frac{1}{2}\int_{\R^3} \nabla\brak{v}^{k} \cdot A[f] \nabla h^2 \dd v  + \int_{\R^3} h^2\nabla \brak{v}^{k} \cdot \nabla a[f] \dd v\\
        &\quad + \frac{1}{2}\int_{\R^3} \brak{v}^{k}\nabla a[f] \cdot \nabla h^2 \dd v + \int_{\R^3} \brak{v}^k h(A[h]:\nabla^2\mu + h\mu) \dd v\\
        &= I_1 + I_2 + I_3 + I_4.
\end{aligned}
\end{equation*}
Recall, by conservation of mass, $\|f\|_{L^\infty_tL^1_v} = 1$ and since $f = h + \mu$, $\|f\|_{L^\infty_{t,v}} \le \eps + \|\mu\|_{L^\infty} \le 2$. It follows from Lemma \ref{lem:Alower} that there is a universal constant such that 
\begin{equation}\label{eq:coeffient_bound}
\|A[f]\|_{L^\infty} \le \|f\|_{L^1}^{2/3}\|f\|_{L^\infty}^{1/3} \le C \qquad \text{and} \qquad \|\nabla a[f]\|_{L^\infty} \le \|f\|_{L^1}^{1/3}\|f\|_{L^\infty}^{2/3} \le C.
\end{equation}
Note that the estimates do not depend on the sign of $f$ and therefore analogous estimates hold for $h$ as well.
From \eqref{eq:coeffient_bound}, we immediately bound $I_2$ as 
\begin{equation*}
|I_2| \le C(k)\int_{\R^3} h^2 \nabla \brak{v}^k \dd v \le C(k)\eps\|h\|_{L^1_k}.
\end{equation*}
For $I_3$, we integrate by parts and similarly obtain from \eqref{eq:coeffient_bound},
\begin{equation*}
|I_3| \le C(k)\int_{\R^3} h^2\abs{\nabla \brak{v}^{k}\cdot \nabla a[f] + \brak{v}^k f} \dd v \le C(k)\eps\norm{h}_{L^1_k}.
\end{equation*}
For $I_1$, we integrate by parts and use \eqref{eq:coeffient_bound} again to obtain
\begin{equation*}
\begin{aligned}
I_1 &= C\int_{\R^3} \nabla\brak{v}^{k} \cdot \left[\nabla \cdot \left(A[f] h^2\right) - \nabla a[f]h^2\right]  \dd v\\
    &\le C(k)\int_{\R^3} \left[\brak{v}^{k-2} + \brak{v}^{k-1}\right]h^2 \dd v \le C(k)\eps \norm{h}_{L^1_k}.
\end{aligned}
\end{equation*}
Finally, for $I_4$, we use \eqref{eq:coeffient_bound} for $h$ and find 
\begin{equation*}
    |I_4| \le C(k)\|A[h]\|_{L^\infty}\int_{\R^3} \brak{v}^k |h|\abs{\nabla^2\mu} \dd v + C(k)\int_{\R^3} h^2\brak{v}^{k}\mu \dd v \le C(k)\eps^{4/3} + C(k) \eps \norm{h}_{L^1_k}.
\end{equation*}
Summarizing, we have shown using the lower bound on $A[f]$ from Lemma \ref{lem:Alower}
\begin{equation*} 
\frac{\dd }{\dd t}\int_{\R^3} \brak{v}^k |h(t)|^2 \dd v + c_0(H)\int_{\R^3} \brak{v}^{k-3}|\nabla h|^2 \dd v \le C(k)\eps\int_{\R^3} \brak{v}^k |h| \dd v.
\end{equation*}
Therefore, integrating in time and using the linear-in-time growth bound on $L^1$-moments of $h$ from Lemma \ref{lem:moments_h}, we obtain
\begin{equation}\label{eqn:L2_estimate2}
\sup_{0 < t < T} \|h\|_{L^2_k}^2 + c_0(H)\int_{0}^{T}\|\nabla h(t)\|_{L^2_{k-3}}^2 \dd t \le C(k)\eps\norm{h}_{L^\infty_t L^1_k} T + \|h_{in}\|_{L^2_k}^2 \le C(\|h_{in}\|_{L^1_{k}}, H)\eps.
\end{equation}
As a second step towards higher regularity, we perform weighted $H^1_k$ estimates by testing \eqref{eq:landau_h} with \\
$\partial_{v_i}\left(\brak{v}^k\partial_{v_i}h\right)$ for $i \in \set{1,2,3}$ and integrating by parts to obtain
\begin{equation*}
\begin{aligned}
    \frac{\dd}{\dd t}\int_{\R^3} \brak{v}^k (\partial_{v_i}h)^2 &+ \int_{\R^3} \brak{v}^kA[f]\nabla\partial_{v_i}h \cdot \nabla\partial_{v_i}h \\
    &\le -\int_{\R^3} \nabla(\brak{v}^{k} \partial_{v_i}h) \cdot \partial_{v_i} A[f] \nabla h \dd v 
     - \int A[f]\partial_{v_i} h\nabla \partial_{v_i} h \nabla \langle v\rangle^k\\
    &\quad+ \int_{\R^3} \nabla (\brak{v}^{k} \partial_{v_i}h) \cdot \partial_{v_i}(\nabla a[f]h) \dd v + \int_{\R^3} \partial_{v_i}\left(\brak{v}^k \partial_{v_i} h\right)\left(A[h]:\nabla^2\mu + h\mu\right) \dd v \\
        &= J_1 + J_2 + J_3 +J_4.
\end{aligned}
\end{equation*}
Before we estimate $J_i, i \in \{1, \dots, 4\}$, we observe that Lemma \ref{lem:Alower} gives the following estimates on the coefficients $A[f]$ and $\nabla a[f]$ 
\begin{equation}\label{eq:coefficient_bound2}
    \|\nabla A[f]\|_{L^\infty_v} \le \|f\|_{L^1}^{1/3}\|f\|_{L^\infty}^{2/3} \le C \qquad \text{and} \qquad \|\nabla^2 a[f]\|_{L^r_v} \le \|f\|_{L^r} \le C \quad\text{where }1 < r < \infty.
\end{equation}
Since these estimates use only the order of decay of the kernel and do not rely upon the sign of $f$, analogous estimates hold for $A[h]$ and $\nabla a[h]$, where the smallness of $h$ gives an improvement (in powers of $\eps$) over the universal constant $C$.
Now, we split $J_1$ as 
\begin{equation*}
    J_1 = \int_{\R^3} \brak{v}^k\nabla \partial_{v_i} h \cdot \partial_{v_i}A[f]\nabla h \dd v + \int_{\R^3} \partial_{v_i}h\nabla\brak{v}^k \partial_{v_i} A[f] \nabla h = J_{11} + J_{12}.
\end{equation*}
For $J_{11}$, we use Young's inequality and \eqref{eq:coefficient_bound2} to conclude
\begin{equation*}
|J_{11}| \le \frac{c_0}{4}\int_{\R^3} \brak{v}^{k-3}|\nabla \partial_{v_i} h|^2 \dd v + C(k,H)\int_{\R^3} |\nabla h|^2 \brak{v}^{k+3}\dd v.
\end{equation*}
For $J_{12}$, we simply use \eqref{eq:coefficient_bound2} to obtain
\begin{equation*}
|J_{12}| \le C(k)\int_{\R^3} |\nabla h|^2 \brak{v}^{k-1} \dd v.
\end{equation*}
For $J_2$ we use \eqref{eq:coeffient_bound} and Young's inequality to get
\beqs
    J_2 \leq  \frac{c_0}{4}\int_{\R^3} \abs{\nabla \partial_{v_i} h}^2 \brak{v}^{k-3} \dd v + C(k,H)\int_{\R^3}\brak{v}^{k+3}\abs{\nabla h}^2\dd v.
\eeqs
Next, we split $J_3$ as
\begin{equation*}
J_3 = \int_{\R^3} \brak{v}^{k} \nabla \partial_{v_i}h \cdot \partial_{v_i}(\nabla a[f]h) \dd v + \int_{\R^3} \partial_{v_i}h\nabla\brak{v}^{k} \cdot \partial_{v_i}(\nabla a[f]h) \dd v = J_{31} + J_{32}.
\end{equation*}
For $J_{31}$, we use Young's inequality, \eqref{eq:coeffient_bound}, H\"older's inequality, and \eqref{eq:coefficient_bound2} with $r = \frac{2(k+6)}{3}$ to obtain
\begin{equation*}
\begin{aligned}
|J_{31}| &\le \frac{c_0}{4}\int_{\R^3} \abs{\nabla \partial_{v_i} h}^2 \brak{v}^{k-3} \dd v + C(k,H)\int_{\R^3}\brak{v}^{k+3} \left(|\nabla\partial_{v_i} a[f]|^2|h|^2 + |\nabla a[f]|^2|\partial_{v_i}h|^2 \right) \\
    &\le \frac{c_0}{4}\int_{\R^3} \abs{\nabla \partial_{v_i} h}^2 \brak{v}^{k-3} \dd v + C(k,H)\eps^{\frac{k+9}{k+6}}\|h\|_{L^1_{k+6}}^{\frac{k+3}{k+6}} + C(k,H) \int_{\R^3} \brak{v}^{k+3}|\nabla h|^2 \dd v.
\end{aligned}
\end{equation*}
For $J_{32}$, we use Young's inequality, \eqref{eq:coeffient_bound}, H\"older's inequality, and \eqref{eq:coefficient_bound2} with $r = \frac{k+6}{3}$ to obtain
\begin{equation*}
\begin{aligned}
|J_{32}| &\le C(k)\int_{\R^3} \brak{v}^{k-1}\left(|\nabla a[f]||\partial_{v_i}h|^2 + |\nabla\partial_{v_i}a[f]||h\partial_{v_i}h| \right) \\
    &\le C(k)\int_{\R^3} \brak{v}^{k-1}\abs{\nabla h}^2 \dd v + C(k)\int_{\R^3} \abs{\nabla h}^2 \brak{v}^{k+3} \dd v + C(k)\eps^{\frac{k+12}{k+6}}\|h\|_{L^1_{k+6}}^{\frac{k}{k+6}}.
\end{aligned}
\end{equation*}

Lastly, we split the linear terms in $J_4$ as
\begin{equation*}
J_4 = \int_{\R^3} \brak{v}^k\partial_{v_iv_i} h \left(A[h]:\nabla^2 \mu + \mu h\right) \dd v + \int_{\R^3} \partial_{v_i}\brak{v}^k\partial_{v_i}h \left(A[h]:\nabla^2\mu + h\mu\right) \dd v = J_{41} + J_{42}.
\end{equation*}
For $J_{41}$, we use Young's inequality, H\"older's inequality, \eqref{eq:coeffient_bound} for $h$, and the embedding $L^2_{3+} \embeds L^1$ so that
\begin{equation*}
\begin{aligned}
|J_{41}| &\le \frac{c_0}{4} \int_{\R^3} \brak{v}^{k-3} \abs{\nabla\partial_{v_i} h}^2 \dd v + C(k,H)\int_{\R^3} \abs{A[h]}^2 \brak{v}^{k+3}\abs{\nabla^2 \mu}^2 \dd v + C(k,H)\int_{\R^3} \brak{v}^{k+3}|\mu|^2 h^2 \dd v\\
    &\le \frac{c_0}{4} \int_{\R^3} \brak{v}^{k-3} \abs{\nabla\partial_{v_i} h}^2 \dd v + C(k,H)\eps^{4/3}\|h\|_{L^1_{3+}}^{2/3} + C(k,H)\eps^2.
\end{aligned}
\end{equation*}
For $J_{42}$, we use Young's inequality, \eqref{eq:coeffient_bound} for $h$, and $L^2_{3+} \embeds L^1$ to obtain
\begin{equation*}
\begin{aligned}
|J_{42}| &\le C(k)\int_{\R^3} \brak{v}^{k+3} \abs{\nabla h}^2 \dd v + C(k)\int_{\R^3} \brak{v}^k\abs{A[h]}^2\abs{\nabla^2 \mu}^2 \dd v + C(k)\int_{\R^3}\brak{v}^k \mu^2 h^2 \dd v\\
    &\le C(k)\int_{\R^3} \brak{v}^{k+3} \abs{\nabla h}^2 \dd v + C(k)\eps^{4/3}\|h\|_{L^1_{3+}}^{2/3} + C(k)\eps^2.
\end{aligned}
\end{equation*}
To summarize, using the lower bound for $A[f]$ from Lemma \ref{lem:Alower} and absorbing the highest order terms without a sign, we find
\begin{equation*}
\begin{aligned}
\frac{\dd}{\dd t}\int_{\R^3} &\brak{v}^k |\partial_{v_i}h|^2 \dd v + \frac{c_0(H)}{4}\int_{\R^3} \brak{v}^{k-3}\abs{\nabla \partial_{v_i}h} \dd v\\
    &\le C(k,H)\int_{\R^3} \brak{v}^{k+3}\abs{\nabla h}^2 \dd v + C(k,H)\eps \left(1 + \int_{\R^3} \brak{v}^{k+6}|h| \dd v\right)^{\max\left(\frac{2}{3}, \frac{k}{k+6},\frac{k+3}{k+6}\right)}.
\end{aligned}
\end{equation*}
After integrating over $(t_1,t_2)$ for $0 < t_1 < t < t_2 < T < 1$ and using the linear-in-time growth bound on $L^1$ moments of $h$ from Lemma \ref{lem:moments_h}, we have
\begin{equation*}
\begin{aligned}
\int_{\R^3} \brak{v}^k \abs{\nabla h(t_2)}^2 \dd v &+ \frac{c_0(H)}{4}\int_{t_1}^{t_2}\int_{\R^3} \brak{v}^{k-3}\abs{\nabla^2 h} \dd v \\
&\le \int_{\R^3} \brak{v}^k \abs{\nabla h(t_1)}^2 \dd v+ C(k,H)\int_{t_1}^{t_2}\int_{\R^3} \brak{v}^{k+3}\abs{\nabla h}^2 \dd v \dd s + C\left(k,H,\|h\|_{L^1_{k+6}}\right)\eps.
\end{aligned}
\end{equation*}
Therefore, taking a supremum over $t_2 \in (t,T)$, averaging over $t_1 \in (0,t)$, appealing to \eqref{eqn:L2_estimate2} with $k$ replaced by $k + 3$, we find
\begin{equation*}
\begin{aligned}
   \sup_{t < s < T} \int_{\R^3} \brak{v}^k \abs{\nabla h}^2 \dd v + \frac{c_0(H)}{4}\int_{t}^{T}\int_{\R^3} \brak{v}^{k-3}\abs{\nabla^2 h} \dd v \dd s \le C\left(k, H, \|h_{in}\|_{L^1_{k+6}}\right)\eps\left(1 + \frac{1}{t}\right),
\end{aligned}
\end{equation*}
which is the claimed estimate.
\end{proof}

\section{Global Existence for Rough Initial Data: Proof of Theorem \ref{thm:existence}}\label{sec:compactness}

In this section, we deduce Theorem \ref{thm:existence} from Proposition \ref{prop:smooth_data} via a compactness argument.

\begin{flushleft}
\underline{{\bf Step 1: Construction of approximating sequence}}
\end{flushleft}

Let us fix $p$, $m$, $H$, and $M$ as in the statement of Theorem \ref{thm:existence}. Then, we fix $\delta_0 = \delta_0(p,m,H,M)$ as defined in Proposition \ref{prop:smooth_data}. For simplicity, dependence of constants on these fixed parameters will be suppressed.

Now, we fix our initial datum $f_{in}:\R^3 \rightarrow \R^+$ as any profile $f_{in} \in L^1_m \cap L^p$ satisfying the normalization \eqref{eq:normalization} (recall this fixes the Maxwellian as $\mu$) and further satisfying the bounds
\begin{equation*}
\int_{\R^3} \brak{v}^m \abs{f_{in}(v) - \mu(v)} \dd v \le M \qquad \text{and} \int_{\R^3} f_{in}\abs{\log f_{in}} \le H.
\end{equation*}
Let us pick a sequence of Schwartz class functions $f_{in}^n \in \S(\R^3)$ that approximate $f_{in}$ and satisfy
\begin{itemize}
    \item Convergence in $L^1_m \cap L^p$, $$\lim_{n\rightarrow \infty} \norm{f_{in}^n - f_{in}}_{L^1_m} + \norm{f_{in}^n - f_{in}}_{L^p} = 0;$$
    \item Normalization \eqref{eq:normalization};
    \item and uniform control of $L^1$-moments, entropy and $L^p$-distance from equilibrium $$\int_{\R^3} \brak{v}^m \abs{f_{in}^n(v) - \mu(v)} \dd v \le M, \quad \int_{\R^3} f_{in}^n\abs{\log f_{in}^n} \le H, \quad \text{and} \quad \int_{\R^3}\abs{f^n_{in} - \mu}^p \dd v \le \delta_0.$$
\end{itemize}
From Proposition \ref{prop:smooth_data}, we obtain a corresponding sequence of global-in-time Schwartz class solutions to \eqref{eq:landau}, namely $f_n:[0,\infty) \times \R^3 \rightarrow \R^+$. 

\begin{flushleft}
\underline{{\bf Step 2: Uniform-in-$n$ estimates}}
\end{flushleft}

We recall that the ODE argument in Lemma \ref{lem:perthame}, the gain of $L^\infty$-regularity in Proposition \ref{prop:degiorgi}, and the $H^1$-regularity estimates in Lemma \ref{lem:higher_regularity} imply the existence of a short time interval $[0,T_0]$ on which we have the following uniform-in-$n$ control:
\begin{equation}\label{eq:short_time_estimates}
\begin{aligned}
    &\sup_{0 < t < T_0} \left( t^{\frac{1}{1+\gamma}}\|f_n(t)\|_{L^\infty(\R^3)}  + \|f_n(t)\|_{L^1_m} + \|f_n(t)\|_{L^p}\right) \le C, \\
    &\sup_{s< t < T_0} \|f_n\|_{H^1_k} \le C, \qquad\text{for each }s\in(0,T_0].
\end{aligned}
\end{equation}
Moreover, the long time behavior from Theorem \ref{thm:desvillettes_he_jiang} combined with the gain of regularity in Proposition \ref{prop:degiorgi} yields the estimates: for each $T_0 < s < \infty$,
\begin{equation}\label{eq:long_time_estimates}
    \sup_{T_0 < t < s} \|f_n(t)\|_{H^1_2} \le C(s) \quad \text{and} \quad \sup_{T_0 < t < s}\|f_n(t)\|_{L^\infty} \le C(s).
\end{equation}
In order to apply the Aubin-Lions lemma and conclude strong compactness of the family $\{f_n\}$, we note that the above control implies corresponding control on the time derivative via standard duality arguments. More precisely, for any $0 < s < t < \infty$ and any for $\Phi \in L^1(s,t,H^1)$,
\begin{equation}
\begin{aligned}
&\abs{\int_s^t\int_{\R^3} \Phi(\tau)\partial_t f_n(\tau) \dd v \dd \tau} \le \abs{\int_s^t\int_{\R^3} \nabla \Phi \cdot \left(A[f_n]\nabla f_n - \nabla a[f_n]f_n\right) \dd v \dd \tau}\\
    &\qquad\le \|\nabla \Phi\|_{L^1(s,t;L^2)}\left(\|A[f_n]\|_{L^\infty_{t,x}}\|\nabla f_n\|_{L^\infty(s,t;L^2)} + \|\nabla a[f_n]\|_{L^\infty_{t,x}} \|f_n\|_{L^\infty(s,t;L^2)}\right)\\
    &\qquad \le C(s,t)\|\Phi\|_{L^1(s,t;H^1)},
\end{aligned}
\end{equation}
where we have used the coefficient bounds in Lemma \ref{lem:Alower}. Therefore,
\begin{equation}\label{eq:long_time_derivative_bound}
\|\partial_t f_n\|_{L^\infty(s,t;H^{-1})} = \sup_{\|\Phi\|_{L^1(s,t;H^1)} = 1} \int_s^t\int_{\R^3}\Phi(\tau)\partial_t f_n(\tau) \dd v\dd \tau \le C(s,t).
\end{equation}
On the other hand, to find a bound on $\partial_t f_n$ which holds all the way up to time $0$, we integrate by parts and obtain for any $\Phi \in L^1(0,T_0;W^{2,p^\prime})$
\begin{equation*}
\int_{\R}\int_{\R^3} \partial_t f_n \Phi \dd v \dd t = \int_{\R}\int_{\R^3} \left(\nabla^2\Phi : A[f_n] f_n\right) + 2\nabla \Phi \cdot \nabla a[f_n]f_n \dd v\dd t.
\end{equation*}
Therefore, using H\"older's inequality, $A[f_n] \in L^\infty(0,T_0;L^\infty)$, $\nabla a[f_n] \in L^\infty(0,T_0;L^3)$, and the Sobolev embedding $\dot{W}^{1,r}(\R^3) \embeds L^{\frac{3r}{3-r}}(\R^3)$ for $1 \le r < 3$, we find
\begin{equation*}
\begin{aligned}
\int_{0}^{T_0}\int_{\R^3} \partial_t f_n \Phi \dd v \dd t &\le \|A[f_n]\|_{L^\infty_{t,v}}\|\nabla^2 \Phi\|_{L^1_tL^{p^\prime}_v}\|f\|_{L^\infty_t L^p_v} + 2\|\nabla \Phi\|_{L^1L^{\frac{3}{2}\frac{p}{p-3/2}}}\|\nabla a[f_n]\|_{L^\infty_tL^3_v}\|f_n\|_{L^\infty_tL^p_v}\\
    &\le C\|\nabla^2 \Phi\|_{L^1_tL^{\frac{p}{p-1}}_v}.
\end{aligned}
\end{equation*}
By duality, this gives the bound
\begin{equation}\label{eq:short_time_derivative_bound}
    \|\partial_t f_n\|_{L^\infty(0,T_0;W^{-2,p})} = \sup_{\norm{\Phi}_{L^1(0,T_0;W^{2,p^\prime})} = 1} \int_0^{T_0} \Phi \partial_t f_n \dd t \le C.
\end{equation}
Finally, from the short time smoothing estimates in \eqref{eq:short_time_estimates}, we see $\{f_n\}$ is bounded in $L^{(1 + \gamma)-}(0,T_0;L^\infty)$. Therefore, interpolation with the $L^\infty(0,T_0;L^1)$ bound yields
\begin{equation}\label{eq:short_time_weak}
\int_0^{T_0} \|f_n(t)\|_{L^2}^2 \dd t \le C.
\end{equation}

\begin{flushleft}
\underline{{\bf Step 3: Construction of the limit}}
\end{flushleft}

First, note that the Banach-Alaoglu theorem in $L^\infty_{loc}(\R^+;H^1_k)$ and $L^2(0,T_0;L^2)$ implies the existence of a weak-star limit $f:\R^+ \times \R^3 \rightarrow \R$ of the sequence $\{f_n\}$. Now, we use a diagonalization argument to combine the uniform-in-estimates from \eqref{eq:short_time_estimates}, \eqref{eq:long_time_estimates}, \eqref{eq:long_time_derivative_bound}, \eqref{eq:short_time_derivative_bound}, and \eqref{eq:short_time_weak} with the Aubin-Lions lemma, Banach-Alaoglu theorem, or Riesz's theorem to pick a subsequence (still denoted $f_n$) converging to $f$ in each of the following senses:
\begin{itemize}
\item Strongly in $C(s,t; L^2)$ for each $0 < s < t < \infty$;
\item weak-starly in $L^\infty(s,t; H^1_k) \cap W^{1,\infty}(s,t; H^{-1})$, for each $0 < s < t < \infty$;
\item weak-starly in $L^\infty(0,T_0; L^p) \cap W^{1,\infty}(0,T_0; W^{-2,p})$;
\item strongly in $C(0,T_0; W^{-1,p})$;
\item weakly in $L^2([0,T_0]\times \R^3)$; 
\item and pointwise almost everywhere in $[0,\infty)\times \R^3$.
\end{itemize}
The pointwise convergence first ensures that $f \ge 0$ almost everywhere. Moreover, combined with Fatou's lemma, $f\in L^\infty(0,\infty;L^1_2) \cap L^\infty_{loc}(0,\infty;L^1_m)$ and $f$ has bounded entropy. Now, for every time $0 \le s \le t < \infty$, we find 
\begin{equation}\label{eq:conservation}
\int_{\R^3} \begin{pmatrix} 1 \\ v \\ |v|^2\end{pmatrix} f_n(s) = \int_{\R^3} \begin{pmatrix} 1 \\ v \\ |v|^2\end{pmatrix} f_n(t).
\end{equation}
Now, recall $f_n \rightarrow f$ pointwise almost everywhere and for almost every $t$, $\|f_n(t)\|_{L^1_m \cap L^2} \le C(t)$. So, for almost every $t$, the fixed time profiles $f_n(t)$ are tight and uniformly integrable and converge pointwise to $f(t)$. The Vitali convergence theorem implies \eqref{eq:conservation} holds for $f$ as well, i.e. $f$ conserves mass, momentum, and energy. Moreover, the same argument implies $f$ has decreasing entropy and satisfies the normalization \eqref{eq:normalization}. Finally, note that the pointwise convergence implies $f$ satisfies the same short-time smoothing estimates as $f_n$.

\begin{flushleft}
\underline{{\bf Step 4: The limit is instantaneously a strong solution to Landau}}
\end{flushleft}

We now show that the limit $f$ satisfies the Landau equation \eqref{eq:landau} in the appropriate sense. Because we have strong compactness in $L^\infty_{loc}(\R^+; L^2)$, passing to the limit in the nonlinear, nonlocal terms is simple. Indeed, for any $\Phi \in L^2(s,t;H^1)$, since $f_n$ solves \eqref{eq:landau}, we find
\begin{equation*}
\int_s^t\int_{\R^3} \Phi \partial_t f_n \dd v \dd \tau= -\int_s^t\int_{\R^3} \nabla \Phi \cdot \left(A[f_n]\nabla f_n - \nabla a[f_n]f_n\right) \dd v \dd \tau.
\end{equation*}
Analyzing the convergence in the top order terms, we have
\begin{equation}\label{eq:Aconvergence1}
\begin{aligned}
    \abs{\int_s^t\int_{\R^3} \nabla \Phi \cdot \left(A[f_n]\nabla f_n - A[f]\nabla f\right) \dd v\dd \tau} &\le \abs{\int_s^t\int_{\R^3} \nabla \Phi \cdot A[f_n - f]\nabla f_n  \dd v\dd \tau}\\
        &\qquad + \abs{\int_s^t\int_{\R^3} \nabla \Phi \cdot A[f]\left(\nabla f_n - \nabla f\right) \dd v\dd \tau}.
\end{aligned}
\end{equation}
For the first term on the right hand side of \eqref{eq:Aconvergence1}, we find
\begin{equation*}
\begin{aligned}
\abs{\int_s^t\int_{\R^3} \nabla \Phi \cdot A[f_n - f]\nabla f_n \dd v\dd \tau} &\le \|\nabla \Phi\|_{L^2_{t,x}}\|\nabla f_n\|_{L^2_{t,x}}\|A[f - f_n]\|_{L^\infty_{t,x}}\\
    &\le C(s,t)\|f-f_n\|_{L^\infty(s,t;L^1)}^{\frac{1}{3}}\|f-f_n\|_{L^\infty(s,t;L^2)}^{\frac{2}{3}},
\end{aligned}
\end{equation*}
which converges to $0$ by conservation of mass and strong compactness in $L^\infty(s,t;L^2)$.
The second term on the right hand side of \eqref{eq:Aconvergence1} converges to $0$ because $f_n \rightarrow f$ weakly in $L^2(s,t;H^1)$ and $A[f]\nabla \Phi \in L^2(s,t;H^1)$.

Next, we analyze the lower order terms, namely,
\begin{equation}\label{eq:aconvergence1}
\begin{aligned}
\bigg|\int_s^t\int_{\R^3} &\nabla \Phi \cdot \left(f_n\nabla a[f_n] - f\nabla a[f]\right) \dd v\dd \tau\bigg|\\
    &\le \int_s^t\int_{\R^3} \abs{\nabla \Phi\cdot (f_n - f)\nabla a[f_n]} \dd v \dd\tau + \int_s^t\int_{\R^3} \abs{\nabla \Phi \cdot f\nabla a[f_n - f]} \dd v \dd\tau\\
    &\le \|\Phi\|_{L^2(s,t;H^1)}\left(\|f_n - f\|_{L^2(t,s;L^2)}\|\nabla a[f_n]\|_{L^\infty_{t,x}} + \|f\|_{L^2_{t,x}}\|\nabla a[f_n -f]\|_{L^\infty_{t,x}}\right)
\end{aligned}
\end{equation}
and the right hand side converges to $0$ as $n \rightarrow \infty$. Finally, using $\partial_t f_n \weak \partial_t f$ in $L^2(s,t;H^{-1})$, we find $f$ satisfies Landau in the sense that for any $0 < s < t < \infty$ and any $\Phi \in L^2(s,t;H^1)$,
\begin{equation}\label{eq:landau_strong}
\int_s^t\int_{\R^3} \Phi \partial_t f \dd v \dd \tau= -\int_s^t\int_{\R^3} \nabla \Phi \cdot \left(A[f]\nabla f - \nabla a[f]f\right) \dd v \dd \tau.
\end{equation}
We have shown $f$ is a weak solution to \eqref{eq:landau} on $(0,\infty)\times \R^3$ in the sense of \eqref{eq:landau_strong}. It remains to show that $f \in C^1(0,\infty;C^2_v)$, so that $f$ is a classical solution for positive times.

However, improving $f\in L^\infty_{t,v}$ to classical regularity for $f$ is relatively standard in the literature. For the sake of completeness, we provide a sketch of one possible argument following \cite{GGZ} and refer the reader to \cite[Lemma 4.2 and Lemma 4.3]{GGZ} to fill in the omitted details.
We note that $f$ belongs to $L^\infty(s,t;L^\infty)$ and the weak formulation \eqref{eq:landau_strong} is sufficiently strong to allow us to use $f$ as a test function. This allows us to perform weighted $H^k$ estimates and obtain for $k \in \R$, $\ell \in \mathbb{N}$, and $0 < s  < t < \infty$:
\begin{equation}\label{eq:higher_regularity}
    \sup_{s < \tau < t} \|\nabla^\ell f\|_{L^2_k}^2 + \int_s^t \|\nabla^{\ell + 1} f\|_{L^2_{k-3}} \dd \tau \le C\left(t,H,\|f_{in}\|_{L^1_{k+6\ell}}, \|f\|_{L^\infty(s,t;L^\infty)}\right) \left(1 + \frac{1}{s}\right)^{\ell}.
\end{equation}
The estimates for $\ell=0$ and $\ell=1$ follow from the proof of Lemma \ref{lem:higher_regularity}, while the estimates for $\ell\ge 2$ are explained thoroughly in \cite[Lemma 4.2]{GGZ}. From \eqref{eq:higher_regularity} with $\ell = 2$ and $k = 0$ and the Sobolev embedding $H^2 \embeds W^{1,6}$, we find
\begin{equation*}
\norm{f}_{L^\infty(s,t;W^{1,6})} \le C\left(s,t,\|f_{in}\|_{L^1_{12}}\right) \qquad \text{and} \qquad \norm{\partial_t f}_{L^\infty(s,t;W^{-1,6})} \le C\left(s,t,\|f_{in}\|_{L^1_{12}}\right),
\end{equation*}
after using a duality argument to bound the time derivative like in \eqref{eq:short_time_derivative_bound}. Real interpolation of (vector-valued) Sobolev spaces implies 
\begin{equation*}
    \norm{f}_{W^{\theta, 6}(s,t;W^{(1-2\theta)-,6})} \le C(s,t), \qquad \text{ for any }\theta \in (0,1).
\end{equation*}
Picking $1/6 < \theta < 1/4$, Morrey's inequality implies $f\in C^{0,\alpha/2}(s,t;C^{0,\alpha})$ for some $\alpha\in (0,1)$. Because $f$ is globally (in velocity) H\"older with exponent $\alpha$, the non-local coefficients $A[f]$ and $\nabla a[f]$ lie in the same H\"older space as $f$. Thus, \eqref{eq:landau} is now a divergence form parabolic equation with H\"older continuous coefficients. Since the lower bound for $A[f]$ still degenerates like $\brak{v}^{-3}$ for large velocities, performing classical parabolic Schauder estimates on compact subsets of $(s,t)\times \R^3$ yields $f \in C^{1,\alpha/2}(s,t;C^{2,\alpha}_{loc})$, which implies $f$ is a classical solution to \eqref{eq:landau}.

\begin{flushleft}
\underline{{\bf Step 5: Behavior at Initial Time}}
\end{flushleft}

It remains only to address sense in which $f$ satisfies \eqref{eq:landau} up to time $t = 0$ and the sense in which the initial datum $f_{in}$ is obtained. For $\varphi \in C^\infty_c([0,\infty) \times \R^3)$, we integrate by parts in \eqref{eq:landau} for $f_n$ and use $\nabla \cdot A = \nabla a$, to obtain
\begin{equation}\label{eq:landau_weakn}
\int_{\R^3} f_n(0)\varphi(0) \dd v - \int_{\R}\int_{\R^3} f_n\partial_t \varphi \dd v \dd t = \int_{\R}\int_{\R^3} \left(\nabla^2\varphi : A[f_n] f_n\right) + 2\nabla \varphi \cdot \nabla a[f_n]f_n \dd v\dd t.
\end{equation}
We use $f_n(0) \rightarrow f_{in}$ in $L^p$ by construction and $f_n \weak f$ in $L^2(0,T;L^2)$ to conclude that the left hand side of \eqref{eq:landau_weakn} converges to
\begin{equation*}
\int_{R^3} f_{in}\varphi(0) \dd v - \int_{\R}\int_{\R^3} f\partial_t \varphi \dd v \dd t.
\end{equation*}
On the other hand, say $\varphi$ is supported in the time interval $[0,t]$. Then, splitting the time integral on the right hand side of \eqref{eq:landau_weakn} into $[0,s]$ and $[s,t]$, and using $f$ is a classical solution to Landau on $[s,t]$, 
we find that for any $s > 0$,
\begin{equation*}
\begin{aligned}
\int_{\R^3} f_{in}\varphi(0) \dd v - \int_{\R}\int_{\R^3} f\partial_t \varphi \dd v \dd t &= \int_s^\infty\int_{\R^3} \left(\nabla^2\varphi : A[f] f\right) + 2\nabla \varphi \cdot \nabla a[f]f \dd v\dd t\\
    &\qquad+ \lim_{n\rightarrow \infty}\int_0^s\int_{\R^3} \left(\nabla^2\varphi : A[f_n] f_n\right) + 2\nabla \varphi \cdot \nabla a[f_n]f_n \dd v\dd t.
\end{aligned}
\end{equation*}
Taking $s\rightarrow 0^+$, it suffices to show that
\begin{equation*}
\lim_{s\rightarrow 0^+} \sup_n \int_0^s\int_{\R^3} \left(\nabla^2\varphi : A[f_n] f_n\right) + 2\nabla \varphi \cdot \nabla a[f_n]f_n \dd v\dd t = 0.
\end{equation*}
First, using Lemma \ref{lem:Alower} and $f_n \weak f$ in $L^2(0,T_0;L^2)$, we have
\begin{equation*}
\int_0^s \nabla^2\varphi : A[f_n]f_n \le s^{1/2}\|\nabla^2\varphi_{L^\infty(0,T_0;L^2)}\|A[f_n]\|_{L^\infty(0,T_0;L^\infty)}\|f_n\|_{L^2(0,T_0;L^2)} \le Cs^{1/2}\|\varphi\|_{L^\infty(0,T_0;H^2)},
\end{equation*}
which evidently converges to $0$ as $s\rightarrow 0^+$, uniformly in $n$. Second, using Lemma \ref{lem:Alower} and $f_n\weak f$ in $L^2(0,T_0;L^2)$ again,
\begin{equation*}
\begin{aligned}
\int_0^s\|\nabla \varphi \cdot \nabla a[f_n]f_n\|_{L^1} \dd t &\le \int_0^s\|\nabla \varphi\|_{L^6}\|\nabla a[f_n]\|_{L^3}\|f_n\|_{L^2} \dd t \le Cs^{1/2}\|\varphi\|_{L^\infty(0,T_0;H^2)},
\end{aligned}
\end{equation*}
which evidently converges to $0$ as $s \rightarrow 0^+$. Thus, by the Lebesgue dominated convergence theorem, $f$ satisfies Landau in the sense that for each $\varphi \in C^\infty_c([0,\infty)\times \R^3)$,
\begin{equation}\label{eq:landau_weak}
\int_{\R^3} f_{in}\varphi(0) \dd v - \int_{\R}\int_{\R^3} f\partial_t \varphi \dd v \dd t = \int_{\R}\int_{\R^3} \left(\nabla^2\varphi : A[f] f\right) + 2\nabla \varphi \cdot \nabla a[f]f \dd v\dd t.
\end{equation}

It remains only to show that the initial datum is obtained strongly in $L^1 \cap L^p$. First, we note that $f \in C(0,T_0; W^{-1,p})$ and so $f$ is certainly continuous in the sense of distributions. Now, taking test functions in \eqref{eq:landau_weak} of the form $\psi(t)\varphi(v)$ for $\psi$ a smooth approximation of $\chi_{[0,s]}(t)$ and $\varphi \in C^\infty_c(\R^3)$, and using the monotone convergence theorem we find for any $0 < s$ and $\varphi\in C^\infty_c(\R^3)$,
\begin{equation}\label{eq:landau_weak2}
\int_{\R^3} f_{in}\varphi \dd v - \int_{\R^3} f(s)\varphi \dd v = \int_{0}^s\int_{\R^3} \left(\nabla^2\varphi : A[f] f\right) + 2\nabla \varphi \cdot \nabla a[f]f \dd v\dd t.
\end{equation}
Since the right hand side of \eqref{eq:landau_weak2} converges to $0$ as $s\to 0$ for any $\varphi\in C^\infty_c$, we conclude that
$$\lim_{s\rightarrow 0^+} \int_{\R^3} \varphi f(s) \dd v = \int_{\R^3} \varphi f_{in}\dd v, \qquad \text{for each }\varphi\in C^\infty_c(\R^3).$$
Because $f$ is continuous in the sense of distributions, $f(0) = f_{in}$, and the initial datum is obtained in the sense of distributions. By a density argument using $f\in L^\infty(0,T_0; L^p)$, the initial datum $f_{in}$ is obtained weakly in $L^p$ so that by weak lower semi-continuity of norms, we obtain
\begin{equation*}
\|f_{in}\|_{L^p} \le \liminf_{s\rightarrow 0^+} \|f(s)\|_{L^p}.
\end{equation*}
Next, analyzing the evolution of the $L^p$ norms of the smooth solutions $f_n$, we find the simple estimate,
\begin{equation*}
\begin{aligned}
\frac{\dd}{\dd t} \int f_n^p &= - \int \nabla f_n^{p-1} \cdot \left(A[f_n]\nabla f_n - \nabla a[f_n]f_n\right) \dd v\\
    &\le -C(p)\int \nabla f_n^{p/2} \cdot A[f_n]\nabla f_n^{p/2} \dd v +  C(p)\int_{\R^3} f_n^{p+1} \dd v\\
    & \le \|f_n\|_{L^p}^p\|f_n\|_{L^\infty}.
\end{aligned}
\end{equation*}
Integrating in time, for each $0\le t \le T_0$,
\begin{equation}\label{eq:short_time_init_data}
\|f_n(t)\|_{L^p}^p \le \|f_n(0)\|_{L^p}^p + \|f_n\|_{L^1(0,t;L^\infty)}\|f_n\|_{L^\infty(0,T_0;L^p)}^p.
\end{equation}
Now taking the limit as $n\rightarrow \infty$ in \eqref{eq:short_time_init_data}, we find for each $0 < t < T_0$,
\begin{equation*}
\|f(t)\|_{L^p}^p \le \|f_{in}\|_{L^p}^p + \|f\|_{L^1(0,t;L^\infty)}\|f\|_{L^\infty(0,T_0;L^p)}^p.
\end{equation*}
Since $f \in L^1(0,T_0;L^\infty)$, taking $t \rightarrow 0^+$, we find
\begin{equation*}
\limsup_{t \rightarrow 0^+} \|f(t)\|_{L^p} \le \|f_{in}\|_{L^p}.
\end{equation*}
Combining $f(t) \weak f_{in}$ in $L^p$  with $\|f(t)\|_{L^p} \rightarrow \|f_{in}\|_{L^p}$ implies strong convergence in $L^p$. Finally, $f\in C(0,T_0;L^p)$ and $f\in L^\infty(0,T_0;L^1_m)$ for $m > 2$ guarantees $f\in C(0,T_0;L^1_2 \cap L^p)$.

\appendix
\section{Weighted Sobolev Inequalities}\label{appendix}
In this section, we provide a brief proof of the Sobolev inequalities claimed in Lemma \ref{lem:poincare}.
Weighted Sobolev and Poincare inequalities of the form
\begin{align*}
    \left( \int_{Q}|\phi|^qw_1\;dv\right)^{1/q}\leq C\left( \int_{Q} |\nabla \phi|^p\;w_2dv\right)^{1/p},
\end{align*} 
for $q\geq p \ge 1$, and $\phi$ either compactly supported in $Q$ or with zero average over $Q$, are guaranteed to hold if certain averages involving $w_1$ and $w_2$ are bounded over all cubes contained in a cube slightly larger than $Q$. For $p=2$ and $2\leq q<\infty$, $r>1$, a cube $Q$, weights $w_1$ and $w_2$, we define
\begin{align*}
    \sigma_{q,2,r}(Q,w_1,w_2) := |Q|^{\frac{1}{d}-\frac{1}{2}+\frac{1}{q}}\left (\fint_{Q}w_1^r\;dv \right )^{\frac{1}{qr}}\left (\fint_{Q}w_2^{-r}\;dv \right )^{\frac{1}{2r}}.
\end{align*}
The precise statement is summarized in the following theorem. For its proof, we refer to \cite[Theorem 1]{SW}.
\begin{theorem}\label{thm:local weighted Sobolev inequalities} 
  Consider two nonnegative  weights $w_1$, $w_2$, a cube $Q$, $2\leq q<\infty$ and $r>1$. Then, for any Lipschitz function $\phi$ which has compact support in  $Q$ or is such that $\int_Q \phi\;dv = 0$, we have
  \begin{align}\label{eqn:Local Weighted Sobolev Inequality General Weights}
    \left( \int_{Q}|\phi|^qw_1\;dv\right)^{1/q}\leq   \mathcal{C}_{q,2,r}(Q,w_1,w_2)\left( \int_{Q}|\nabla \phi|^2\; w_2\;dv\right)^{1/2},
  \end{align} 
  where for some constant $C(d,r,q)$,
  \begin{align*}
    \mathcal{C}_{q,2,r}(Q,w_1,w_2) := C(d,r,q)\sup \limits_{Q' \subset 8Q} \sigma_{q,2,r}(Q',w_1,w_2).
  \end{align*}
  \end{theorem}
We are now prepared to show Lemma \ref{lem:poincare} by combining Theorem \ref{thm:local weighted Sobolev inequalities} with a covering argument.
\begin{flushleft}
{\bf \underline{Proof of Lemma \ref{lem:poincare}}}
\end{flushleft}
Let $Q = Q(v_0)$ be a cube of side length $1$, centered at $v_0$. For any $\alpha,\; m > 0$, we compute 
\begin{align*}
  \frac{1}{|Q|}\int_{Q} {(1+|v|)^{m}}\;\dd v  \leq 
  \left \{ \begin{array}{rl}
 2^m (1+|v_0|)^m ,  & v_0\in {B_{2}(0)}^c,\\	 
   {}\\
   c,  \hspace{1cm}& v_0\in {B_{2}(0)},
   \end{array}\right.
  \end{align*}
  and 
   \begin{align*}
  \frac{1}{|Q|}\int_{Q} \frac{1}{(1+|v|)^\alpha}\;\dd v  \leq 
  \left \{ \begin{array}{rl}
   \frac{2}{(1+|v_0|)^\alpha},& v_0\in {B_{2}(0)}^c,\\
   {}\\
   c,  \hspace{1cm}& v_0\in {B_{2}(0)}.
   \end{array}\right.
  \end{align*}
Therefore,
\begin{align*}
  \sigma_{6,2,r}(Q,\langle v \rangle^{-9},\langle v \rangle^{-3}) = \left (\fint_{Q}\langle v \rangle^{-9r}\;\dd v \right )^{\frac{1}{6 r}}\left (\fint_{Q}\langle v \rangle^{3r}\;\dd v \right )^{\frac{1}{2r}} \le C,
\end{align*}
where $C$ is independent on the center of $Q$. Applying Theorem \ref{thm:local weighted Sobolev inequalities}, we find (\ref{eqn:Local Weighted Sobolev Inequality General Weights}) holds for $w_1  = \langle v \rangle^{-9}$ and $w_2 =\langle v \rangle^{-3}$:
\begin{align}\label{eqn:Local Weighted Sobolev Inequality Landau weights}
    \left( \int_{Q}|\phi - (\phi)_Q|^6 \langle v \rangle^{-9}\;\dd v\right)^{1/3}\leq   C \int_{Q}|\nabla \phi|^2\;  \langle v\rangle^{-3}\;\dd v.
\end{align} 
Here and below $(\phi)_Q$ denotes the average of $\phi$ in $Q$. 

Next, we consider a family $\mathcal{Q}$ of non-overlapping cubes of side length $1$ which cover $\mathbb{R}^3$. Decomposing the integral over $\R^3$ using this partition and Young's inequality, we have
\begin{align*}	
    \int_{\mathbb{R}^3}|\phi|^6\brak{v}^{-9}\;dv = \sum \limits_{Q\in\mathcal{Q}} \int_{Q}|\phi|^6 \brak{v}^{-9}\;dv
	&\leq c_1\sum \limits_{Q\in\mathcal{Q}} \int_{Q}|\phi-(\phi)_{Q}|^6 \brak{v}^{-9}\;\dd v + c_2\sum \limits_{Q\in\mathcal{Q}} |(\phi)_{Q}|^6 \int_{Q}\brak{v}^{-9}\;\dd v. 
\end{align*}
For each $Q\in\mathcal{Q}$, we apply \eqref{eqn:Local Weighted Sobolev Inequality Landau weights} and find 
\begin{align*}
    \sum \limits_{Q\in\mathcal{Q}} \int_{Q}|\phi-(\phi)_{Q}|^6 \brak{v}^{-9}\;\dd v &\le \sum \limits_{Q\in\mathcal{Q}} \int_{Q}|\phi-(\phi)_{Q}|^6 \brak{v}^{-9}\;\dd v\\
        &\leq  C \sum \limits_{Q\in\mathcal{Q}} \left(\int_{Q} |\nabla \phi|^2\brak{v}^{-3}\;\dd v \right)^{3}
        \le C \left (\int_{\mathbb{R}^3}|\nabla \phi|^2\brak{v}^{-3}\;\dd v \right )^{3}
\end{align*}
The last inequality follows by choice of $Q$
\begin{equation}\label{eq:sums}
\begin{aligned}
      \sum \limits_{Q\in\mathcal{Q}} \left(\int_{Q} |\nabla \phi|^2\brak{v}^{-3}\;\dd v \right)^{3} &\leq \left(\int_{\R^3} |\nabla \phi|^2\brak{v}^{-3}\;\dd v \right)^{2}\left( \sum \limits_{Q\in\mathcal{Q}} \int_{Q} |\nabla \phi|^2\brak{v}^{-3}\;\dd v\right) \\
      &=\left(\int_{\R^3} |\nabla \phi|^2\brak{v}^{-3}\;\dd v \right)^{2} \left(\int_{\R^3} |\nabla \phi|^2\brak{v}^{-3}\;\dd v \right).
\end{aligned}
\end{equation}
Additionally, since $\brak{v}^{-9} \le 1$ and $|Q| = 1$, for any $1\le s \le 6$, Jensen's inequality implies
\begin{align*}
\sum \limits_{Q\in\mathcal{Q}} |(\phi)_{Q}|^6 \int_{Q}\brak{v}^{-9}\;\dd v \le  \sum \limits_{Q\in\mathcal{Q}} |(\phi)_{Q}|^6  \le \sum \limits_{Q\in\mathcal{Q}} \left(\int_Q |\phi|^s \;\dd v\right)^{6/s} \le  \left( \int_{\mathbb{R}^3}  \abs{\phi}^s \;\dd v \right)^{6/s}.
\end{align*}
The last inequality follows by the same argument as in \eqref{eq:sums}. 
Summarizing, we obtain for any $1 \le s \le 6$,
\begin{align*}
    \int_{\mathbb{R}^3}|\phi|^6 \brak{v}^{-9}\;\dd v  \le  C_1 \left (\int_{\mathbb{R}^3}|\nabla \phi|^2 \brak{v}^{-3}\;\dd v \right )^{3} + C_2\left(  \int_{\mathbb{R}^3}  \abs{\phi}^s \;\dd v \right)^{6/s},
\end{align*}
as claimed.




\bibliographystyle{plain}
\bibliography{Landau}







\end{document}